\documentclass[11pt,twoside,letter]{article}
\usepackage{amsthm, amsmath, amssymb, amsfonts, graphicx, epsfig}
\usepackage{accents}
\usepackage{bbm}
\usepackage{mathrsfs}
\usepackage{ulem}
\usepackage{cancel}
\usepackage{epstopdf}
\usepackage{graphicx}
\usepackage[font=sl,labelfont=bf]{caption}
\usepackage{subcaption}
\usepackage{float}
\restylefloat{table}
\usepackage{enumerate}
\usepackage[usenames,dvipsnames]{color}

\definecolor{labelkey}{rgb}{0,0,1}
\usepackage[colorlinks=true, pdfstartview=FitV, linkcolor=blue,
            citecolor=blue, urlcolor=blue]{hyperref}
\usepackage[usenames]{color}
\usepackage{url}
\makeatletter\def\url@leostyle{%
 \@ifundefined{selectfont}{\def\UrlFont{\sf}}{\def\UrlFont{\scriptsize\ttfamily}}} \makeatother
\usepackage[framemethod=default]{mdframed}
\usepackage{array}
\usepackage[margin=1.0in, letterpaper]{geometry}
\newtheorem{theorem}{Theorem}
\newtheorem{proposition}[theorem]{Proposition}
\newtheorem{lemma}[theorem]{Lemma}

\newtheorem{assumption}[theorem]{Assumption}
\theoremstyle{definition}

\theoremstyle{remark}
\newtheorem{remark}[theorem]{Remark}

\numberwithin{equation}{section}
\numberwithin{theorem}{section}
\definecolor{Red}{rgb}{1.0,0,0.0}
\definecolor{Blue}{rgb}{0,0.0,1.0}
\def\cB{\mathcal{B}}

\def\cM{\mathcal{M}}

\def\cP{\mathcal{P}}
\def\cQ{\mathcal{Q}}

\def\bE{\mathbb{E}}
\def\bF{\mathbb{F}}

\def\bN{\mathbb{N}}

\def\bP{\mathbb{P}}

\def\bR{\mathbb{R}}

\def\sD{\mathscr{D}}

\def\sF{\mathscr{F}}
\def\sG{\mathscr{G}}

\def\sX{\mathscr{X}}
\def\sY{\mathscr{Y}}
\def\sZ{\mathscr{Z}}

\def\mI{\mathsf{I}}
\def\mJ{\mathsf{J}}

\def\mQ{\mathsf{Q}}

\def\mU{\mathsf{U}}
\def\mV{\mathsf{V}}

\newcommand{\wt}{\widetilde}
\newcommand{\wh}{\widehat}

\newcommand{\1}{\mathbbm{1}}            % preferable way of writing indicator function
\newcommand{\set}[1]{\{#1\}}            % set: {xyz} to be used for inline formulas
\DeclareMathOperator{\dif}{d \!}        % used for differential, same as in commath.sty
\DeclareMathOperator{\supp}{supp}          % support

\title{The two-sided exit problem for an additive functional of a time-inhomogeneous Markov chain}

\author{
    Tomasz R. Bielecki %\\[-0.3ex]
	\thanks{Department of Applied Mathematics, Illinois Institute of Technology, USA. \textbf{Email:} tbielecki@iit.edu}
 \and
    Ziteng Cheng %\\[-0.3ex]
	\thanks{Department of Statistical Sciences, University of Toronto, Canada. \textbf{Email:} ziteng.cheng@utoronto.ca}
 \and
     Ruoting Gong %\\[-0.3ex]
	\thanks{Mathematical Reviews, American Mathematical Society, USA. \textbf{Email:} rxg@ams.org}
        }
\date{\today}

\begin{document}

\maketitle

\begin{abstract}\noindent
We consider an additive functional driven by a time-inhomogeneous Markov chain with a finite state space. Our study focuses on the joint distribution of the two-sided exit time and the state of the driving Markov chain at the time of exit, given in terms of expectation operators. These operators can be expressed as compositions of other operators related to some relevant one-sided exit (or first passage) problems. In addition, we study the law of the driving Markov chain at times prior to the exit time.
\end{abstract}

\section{Introduction}

In this paper we consider a time-inhomogeneous Markov process {$X=(X_t)_{t\ge0}$} with finite state space $\mathbf E$, and an additive functional
\begin{align}\label{eq:Defphi0}
\phi_{t}(s):=\int_{s}^{t}v(X_{u})\,du,\quad t\in[s,\infty),
\end{align}
where $s\geq 0$ and $v\ne 0$ is a real valued function defined on $\mathbf E$. For $\ell\ge 0$, we define the following first passage times
\begin{align}
\tau_{\ell}^{+}(s):=\inf\left\{t\in[s,\infty]:\,\phi_{t}(s)>\ell\right\}\quad\text{and}
\quad\tau_{\ell}^{-}(s):=\inf\left\{t\in[s,\infty]:\,\phi_{t}(s)<-\ell\right\}.
\end{align}
Our goal is to provide a analytical/algebraic expression for the following expectation
\begin{align}\label{eq:MainProblem}
\bE_{s,i}\left(g\left(\tau^{-}_{\ell^-}(s)\wedge\tau^{+}_{\ell^+}(s), X_{\tau^{-}_{\ell^-}(s)\wedge\tau^{+}_{\ell^+}(s)}\right)\right)
\end{align}
for any $s\ge 0$, $i\in\mathbf E$, $\ell^+,\ell^-\ge 0$ and bounded $g:\bR_+\times\mathbf E\to\bR$. This is done via \eqref{eq:XipmSum} and Theorem \ref{thm:TwoSidedExit}.

%To illustrate the purpose of this study, ........... we argue that, while it is possible to use the Feynman-Kac formulation (cf. \cite[Section 4.4]{KaratzasShreve1998}) to compute the expectation in \eqref{eq:MainProblem} for one $g$ \trb{if for one g, then for a wide range of g's as well; so I am not sure what we are trying to communicate here}, it is prudent to develop a method that computes the expectation for a wide range of $g$'s, or even, the law of the pair  $(\tau^{-}_{\ell^-}(s)\wedge\tau^{+}_{\ell^+}(s), X_{\tau^{-}_{\ell^-}(s)\wedge\tau^{+}_{\ell^+}(s)})$  \trb{we are not really doing this; our expectations allow to approximate the this law though}.
%
%\textcolor{Brown}{If one needs to compute the expectation in \eqref{eq:MainProblem} for a wide range of $g$'s, computing the expectation for each $g$ separately using Feynman-Kac could be computationally intensive. The method we propose here potentially leads to method that computes the expectation in \eqref{eq:MainProblem} for various efficiently, or even, approximate  the law of the pair  $(\tau^{-}_{\ell^-}(s)\wedge\tau^{+}_{\ell^+}(s), X_{\tau^{-}_{\ell^-}(s)\wedge\tau^{+}_{\ell^+}(s)})$.}

For the case of $\ell^-=-\infty$ (i.e., one-sided exit, or, first passage), under suitable conditions, one can derive equations that compute $\bE(g(\tau^{+}_{\ell^+}(s), X_{\tau^{+}_{\ell^+}(s)}))$ for various $g$'s. Most of the literature considers time-homogeneous setting. \cite{BarlowRogersWilliams1980} and \cite{KennedyWilliams1990} respectively study the one-sided exit problems without and with additive Brownian noise. That is, respectively, the case of
\begin{align}\label{eq:Defphi0-1}
\varphi_{t}:=\int_{0}^{t}v(X_{u})\,du,\quad t\in[0,\infty),
\end{align} and the case of \begin{align}\label{eq:Defphi0-2}
\widehat \varphi_{t}:=\int_{0}^{t}v(X_{u})\,du +W_t,\quad t\in[0,\infty),
\end{align} where $(W_{t})_{t\geq 0}$ is a standard Brownian motion independent of $X$. In \cite{Rogers1994}, alternative proofs of the main results in \cite{BarlowRogersWilliams1980} and \cite{KennedyWilliams1990} are provided. The proofs are facilitated by introducing certain novel martingales (cf. \cite[(2.10) and (7.2)]{Rogers1994}). Using similar martingales, \cite{JiangPistorius2008} generalizes the previous results to allow coexistence of unnoisy and noisy regions (we refer to \cite[Sections 3 and 4]{JiangPistorius2008} for details). Recently, \cite{BieCheCiaGon2020} studied the one-sided exit problem and computation of $\bE_{s,i}\left (g\left(\tau^{+}_{\ell^+}(s), X_{\tau^{+}_{\ell^+}(s)}\right)\right )$  in a time-inhomogeneous setting. Although martingales analogous to the aforementioned martingales are not explicitly used there, the proof in \cite{BieCheCiaGon2020} is built upon a similar martingale idea.

However, the martingale method described above for the one-sided exit problems is not readily applicable to the two-sided exit problem here. Therefore, in this paper, we will express \eqref{eq:MainProblem} in terms of some operators related to $\tau^\pm_{\ell^\pm}$, which are motivated by studies done in \cite{BieCheCiaGon2020}. The analogous idea
can also be found in \cite{JiangPistorius2008} under time-homogeneous setting. However, some of the techniques used in \cite{JiangPistorius2008} are not quite compatible with time-inhomogeneity, thus we adopt a different approach. Our approach is mainly based on the probabilistic decomposition such as \eqref{eq:XiRewrite} and its analytic counterpart \eqref{eq:JPminusDecomp}, which originate from \eqref{eq:XipmSum} and a reformulation of the problem in Section \ref{subsec:Markov}.

It is important to acknowledge the fact that the two-sided exit problem is used in  various applications. For example, in \cite{JiangPistorius2008} the solution of a two-sided exit problem was used as a tool for dealing with pricing a perpetual American option subject to a stock dynamics that is modulated by a Markov driver. In \cite{EbertStrack2015} a two-sided exit problem was used a tool in studying issues arising in so called prospect theory. In \cite{StechmannNeelin2014}  a two-sided exit problem was considered in relation to precipitation statistics. Numerous applications in physics of two-sided exit problems are discussed in \cite{Redner2001}.

The rest of the paper is organized as below. We first introduce our setup in Section \ref{sec:Setup}. In Section \ref{sec:MainResults} we present the main result of this paper, Theorem \ref{thm:TwoSidedExit}. Then, we give Proposition \ref{prop:AfterExit} which demonstrates that our main result can also be used to compute expectation of a function the driving state at times prior to the two-sided exit. The proofs are gathered in Section \ref{sec:Proofs}. Finally,  in Section \ref{sec:Conclusion}, we close with some concluding remarks and suggestions for a follow-up research.

\section{Setup}\label{sec:Setup}

Our setup is similar to that of \cite{BieCiaGonHua2020}, except that some assumptions on the regularity of the generator of the underlying Markov chain are relaxed here. The rest of the section details our setup.

\subsection{Preliminaries}\label{subsec:Notations}

Throughout this paper we let $\mathbf{E}$ be a finite set, with $|\mathbf{E}|=m>1$. We define $\overline{\mathbf{E}}:=\mathbf{E}\cup\{\partial\}$, where $\partial$ denotes the coffin state isolated from $\mathbf{E}$. Let $(\mathsf{\Lambda}_{s})_{s\in\bR_{+}}$, where $\bR_{+}:=[0,\infty)$, be a family of $m\times m$ sub-Markovian generator matrices, i.e., their off-diagonal elements are non-negative, and the entries in their rows sum to a non-positive number. We additionally define $\mathsf{\Lambda}_{\infty}:=\mathsf{0}$, the $m\times m$ matrix with all entries equal to zero.

\medskip
We make the following standing assumption:
\begin{assumption}\label{assump:GenLambda}\mbox{}
There exists an absolute constant $K\in(0,\infty)$, such that $|\mathsf{\Lambda}_{s}(i,j)|\leq K$, for all $i,j\in\mathbf{E}$ and $s\in\bR_{+}$.
\end{assumption}
Let $v:\overline{\mathbf{E}}\rightarrow\bR$ with $v(i)\neq 0$ for any $i\in\mathbf{E}$ and $v(\partial)=0$. %$\mV:=\text{diag}\{v(i):i\in\mathbf{E}\}$, $\overline{v}:=\max_{i\in\mathbf{E}}|v(i)|$, and $\underline{v}:=\min_{i\in\mathbf{E}}|v(i)|$.
We will use the following partition of the set $\mathbf{E}$
\begin{align*}
\mathbf{E}_{+}:=\left\{i\in\mathbf{E}:\,v(i)>0\right\}\quad\text{and}\quad\mathbf{E}_{-}:=\left\{i\in\mathbf{E}:\,v(i)<0\right\}.
\end{align*}
We assume that both $\mathbf{E}_{+}$ and $\mathbf{E}_{-}$ are non-empty. Without loss of generality, we also assume that the indices of the first $m_+=|\mathbf{E}_{+}|$ (respectively, last $m_-=|\mathbf{E}_{-}|$) rows and columns of any $m\times m$ matrix correspond to the elements in $\mathbf{E}_{+}$ (respectively, $\mathbf{E}_{-}$).

In what follows we let $\sX:=\bR_{+}\times\mathbf{E}$, and  $\sX_{\pm}:=\bR_{+}\times\mathbf{E}_{\pm}$. The Borel $\sigma$-field on $\sX$ (respectively, $\sX_{\pm}$) is denoted by $\cB(\sX):=\cB(\bR_{+})\otimes 2^{\mathbf{E}}$ (respectively, $\cB(\sX_{\pm}):=\cB(\bR_{+})\otimes 2^{\mathbf{E}_{\pm}}$). Accordingly, we let $\overline{\sX}:=\sX\cup(\infty,\partial)$ (respectively, $\overline{\sX_{\pm}}:=\sX_{\pm}\cup(\infty,\partial)$) be the one-point completion of $\sX$ (respectively, $\sX_{\pm}$), and let $\cB(\overline{\sX}):=\sigma(\cB(\sX)\cup\{(\infty,\partial)\})$ (respectively, $\cB(\overline{\sX_{\pm}}):=\sigma(\cB(\sX_{\pm})\cup\{(\infty,\partial)\})$). A pair $(s,i)\in \sX$ consists of the time variable $s$ and the space variable $i$.

\medskip
We will also use the following notations for various spaces of real-valued functions:
\begin{itemize}
\item $\cB_b(\overline{\sX})$ is the space of $\cB(\overline{\sX})$-measurable, and bounded functions $f$ on $\overline{\sX}$, with $g(\infty,\partial)=0$.
\item $C_{0}(\overline{\sX})$ is the space of functions $f\in \cB_b(\overline{\sX})$ such that $f(\cdot,i)\in C_{0}(\bR_{+})$ for all $i\in\mathbf{E}$, where $C_{0}(\bR_{+})$ is the space {of} functions vanishing at infinity.
\item $C_{c}(\overline{\sX})$ is the space of functions $f\in \cB_b(\overline{\sX})$ such that $f(\cdot,i)\in C_{c}(\bR_{+})$ for all $i\in\mathbf{E}$, where $C_{c}(\bR_{+})$ is the space of functions with compact support.
\end{itemize}

Sometimes $\overline{\sX}$ will be replaced by $\overline{\sX_{+}}$ or $\overline{\sX_{-}}$ when the functions are defined on these spaces, in which case the set $\mathbf{E}$ will be replaced by $\mathbf{E}_{+}$ or $\mathbf{E}_{-}$, respectively, in the above definitions.  Note that each function on $\overline{\sX}$ can be viewed as a time-dependent vector of size $m$, which can be split into a time-dependent vector of size $m_{+}$ (a function on $\sX_{+}$) and a time-dependent vector of size $m_{-}$ (a function on $\sX_{-}$).

\subsection{A time-inhomogeneous Markov family corresponding to  $(\mathsf{\Lambda}_{s})_{s\in\bR_{+}}$ and related passage times}\label{subsec:Markov}

We start by introducing a time-inhomogeneous Markov Family corresponding to sub-Markovian matrix intensity function $(\mathsf{\Lambda}_{s})_{s\in\bR_{+}}$. %\sout{We say {that} $(\mathsf{\Lambda}_{s})_{s\in\bR_{+}}$ is a sub-Markovian matrix intensity function if $[\mathsf{\Lambda}_s]_{ii}\le 0$, $[\mathsf{\Lambda}_s]_{ij}\ge 0$ and $\sum_{j\in\mathbf E}[\mathsf{\Lambda}_s]_{ij}\le 0$ for all $(s,i,j)\in\bR_+\times\mathbf E^2$. We say {that} $(\mathsf{\Lambda}_{s})_{s\in\bR_{+}}$ is a Markovian matrix intensity function if it additionally satisfies $\sum_{j\in\mathbf E}[\mathsf{\Lambda}_s]_{ij}=0$ for all $(s,i,j)\in\bR_+\times\mathbf E^2$. A sub-Markovian matrix intensity function can be transformed into a Markovian matrix intensity function by introducing an absorption state (also called coffin state).} \sout{We refer to} \cite[Section 8.4.2]{RolskiSchmidliSchmidtTeugels1999} \sout{for more details on Markovian matrix intensity function and the associated evolution system.} \trb{We already said in the previous section what sub-Markovian means. We do not need to say what Markovian means! We do not have any results specifically for Markovian case.}
After the introduction of the time-inhomogeneous Markov Family, we proceed with a study of some passage times related to this family.

\subsubsection{A time-inhomogeneous Markov family $\cM$ corresponding to $(\mathsf{\Lambda}_{s})_{s\in\bR_{+}}$} \label{subsubsec:Markov}

We take $\Omega$ as the collection of $\overline{\mathbf{E}}$-valued functions $\omega$ on $\bR_{+}$, and $\sF:=\sigma\{X_{t},\,t\in\bR_{+}\}$, where $X$ is the coordinate mapping $X_\cdot(\omega):=\omega(\cdot)$. Sometimes we may need the value of $\omega\in\Omega$ at infinity, and in such case we set $X_{\infty}(\omega)=\omega(\infty)=\partial$, for any $\omega\in\Omega$. We endow the space $(\Omega,\sF)$ with a family of filtrations $\bF_{s}:=\{\sF^{s}_{t},\,t\in[s,\infty]\}$, $s\in\overline{\bR}_{+}$, where, for $s\in\bR_{+}$,
\begin{align*}
\sF^{s}_{t}:=\bigcap_{r>t}\sigma\left(X_{u},\,\,u\in[s,r]\right),\,\,\,t\in[s,\infty);\quad\sF^{s}_{\infty}:=\sigma\bigg(\bigcup_{t\geq s}\sF^{s}_{t}\bigg),
\end{align*}
and $\sF^{\infty}_{\infty}:=\{\emptyset,\Omega\}$. We denote by
\begin{align*}
\mathcal{M}:=\big\{\big(\Omega,\sF,\bF_{s},(X_t)_{t\in[s,\infty]},\bP_{s,i}\big),\,(s,i)\in\overline{\sX}\big\}
\end{align*}
a canonical {\it time-inhomogeneous} Markov family. That is,
\begin{itemize}
\item $\bP_{s,i}$ is a probability measure on $(\Omega,\sF^{s}_{\infty})$ for $(s,i)\in\overline{\sX}$;
\item the function $P:\overline{\sX}\times\overline{\bR}_{+}\times 2^{\overline{\mathbf{E}}}\rightarrow [0,1]$ defined for $0\leq s\leq t\leq\infty$ as
    \begin{align*}
    P(s,i,t,B):=\bP_{s,i}\!\left(X_t\in B\right)
    \end{align*}
    is measurable with respect to $i$ for any fixed $s\leq t$ and $B\in 2^{\overline{\mathbf{E}}}$;
\item $\bP_{s,i}(X_{s}=i)=1$ for any $(s,i)\in\overline{\sX}$;
\item for any $(s,i)\in\overline{\sX}$, $s\leq t\leq r\leq\infty$, and $B\in 2^{\overline{\mathbf{E}}}$, it holds that
    \begin{align*}
    \bP_{s,i}\!\left(X_r\in B\,|\,\sF^{s}_{t}\right)=\bP_{t,X_{t}}\!\left(X_r\in B\right),\quad\bP_{s,i}-\text{a.s.}\,,
    \end{align*}
%\item \textcolor{Brown}{$X$ is $\bP_{s,i}$ almost surely c\`adl\`ag for any $(s,i)\in\overline{\sX}$.}
\end{itemize}

Let $\mU:=(\mU_{s,t})_{0\leq s\leq t<\infty}$ be the evolution system (cf. \cite{Bottcher2014}) corresponding to $\mathcal{M}$, defined by
\begin{align}\label{eq:DefEvolSytXStar}
\mU_{s,t}f(i):=\bE_{s,i}\left(f(X_{t})\right),\quad 0\leq s\leq t<\infty,\quad i\in\mathbf{E},
\end{align}
for all functions (column vectors) $f:\mathbf{E}\rightarrow\bR$.\footnote{Note that for $t\in\bR_{+}$, $X_{t}$ takes values in $\mathbf{E}$.} We assume that
\begin{align}\label{eq:DefGenXStar}
\lim_{h\downarrow 0}\frac{1}{h}\left(\mU_{s,s+h}f(i)-f(i)\right)=\mathsf{\Lambda}_{s}f(i),\quad\text{for any }\,(s,i)\in\sX,
\end{align}
for all $f:\mathbf{E}\rightarrow\bR$. Without loss of generality,\footnote{We refer to the discussion in \cite[Section 2.2.1]{BieCheCiaGon2020}.} we assume $\mathcal{M}$ is a standard Markov family (cf. \cite[Definition I.6.6]{GikhmanSkorokhod2004}).\footnote{In terminology of \cite[Definition I.6.6]{GikhmanSkorokhod2004}, $\cM$ is called a Markov process instead of a Markov family. } In particular, $\cM$ has {the} strong Markov property. Consequently, each process $(X_t)_{t\in[s,\infty]}$, for $s\in\bR_+$, is a strong Markov process.

\subsubsection{Passage times related to $\cM$}\label{subsubsec:PassageTime}

For all $s\in\overline\bR_+$ and $\omega\in\Omega$, we define an additive functional $\phi_{\cdot}(s)$ as
\begin{align}\label{eq:Defphi}
\phi_{t}(s):=\int_{s}^{t}v(X_{u})\,du,\quad t\in[s,\infty).
\end{align}
Additionally, we define $\phi_\infty(s,\omega):=0$. %we stipulate $\phi_{\infty}(s,\omega)=0$ for any $s\in[0,\infty]$ for every $\omega\in\Omega$.
{Moreover}, for any $s\in\overline{\bR}_{+}$ and $\ell\in\bR_{+}$, we define associated passage times
\begin{align}\label{eq:Deftau}
\tau_{\ell}^{+}(s):=\inf\left\{t\in[s,\infty]:\,\phi_{t}(s)>\ell\right\}\quad\text{and}
\quad\tau_{\ell}^{-}(s):=\inf\left\{t\in[s,\infty]:\,\phi_{t}(s)<-\ell\right\}.
\end{align}
Both $\tau_{\ell}^{+}(s)$ and $\tau_{\ell}^{-}(s)$ are $\bF_{s}$-stopping times since, $\phi_{\cdot}(s)$ is $\bF_{s}$-adapted, has continuous sample paths, and $\bF_{s}$ is right-continuous (cf. \cite[Proposition 1.28]{JacodShiryaev2003}). For notational convenience, if no confusion arises, we will omit the parameter $s$ in $\phi_{t}(s)$ and $\tau_{\ell}^{\pm}(s)$.

The following lemma is an immediate consequence of the
definitions above. We refer to \cite[Lemma 2.2]{BieCheCiaGon2020} for the proof.
\begin{lemma}\label{lem:RangeXTaupm}
For any $s\in\overline{\bR}_{+}$, $\ell\in\bR_{+}$ and $\omega\in\Omega$ the following inclusion holds $X_{\tau_{\ell}^{\pm}(s)}(\omega)\in\mathbf{E}_{\pm}\cup\{\partial\}$. In particular, if $\tau_{\ell}^{\pm}(s,\omega)<\infty$, then $X_{\tau_{\ell}^{\pm}(s)}(\omega)\in\mathbf{E}_{\pm}$.
\end{lemma}

In order to proceed, we introduce the following operators:
\begin{itemize}
\item $J^{+}:\cB_b(\overline{\sX_{+}})\rightarrow \cB_b(\overline{\sX_{-}})$ is defined as
    \begin{align}\label{eq:DefJPlus}
    \big(J^{+}g^{+}\big)(s,i):=\bE_{s,i}\Big(g^{+}\Big(\tau_{0}^{+},X_{\tau_{0}^{+}}\Big)\Big),\quad (s,i)\in\overline{\sX_{-}}.
    \end{align}
    Clearly, for any $g^{+}\in \cB_b(\overline{\sX_{+}})$ it holds that $|(J^{+}g^{+})(s,i)|\leq\|g^{+}\|_{\cB_b(\overline{\sX_{+}})}<\infty$ for any $(s,i)\in\sX_{-}$, and $(J^{+}g^{+})(\infty,\partial)=0$, so that $J^{+}g^{+}\in \cB_b(\overline{\sX_{-}})$.
\item $J^{-}:\cB_b(\overline{\sX_{-}})\rightarrow \cB_b(\overline{\sX_{+}})$ is defined as,
    \begin{align}\label{eq:DefJMinus}
    \big(J^{-}g^{-}\big)(s,i):=\bE_{s,i}\Big(g^{-}\Big(\tau_{0}^{-},X_{\tau_{0}^{-}}\Big)\Big),\quad (s,i)\in\overline{\sX_{+}}.
    \end{align}
\item For any $\ell\in\bR_{+}$, $\cP_{\ell}^{+}:\cB_b(\overline{\sX_{+}})\rightarrow \cB_b(\overline{\sX_{+}})$ is defined as
    \begin{align}\label{eq:DefPPlus}
    \big(\cP^{+}_{\ell}g^{+}\big)(s,i):=\bE_{s,i}\Big(g^{+}\Big(\tau_{\ell}^{+},X_{\tau_{\ell}^{+}}\Big)\Big),\quad (s,i)\in\overline{\sX_{+}}.
    \end{align}
\item For any $\ell\in\bR_{+}$, $\cP_{\ell}^{-}:\cB_b(\overline{\sX_{-}})\rightarrow \cB_b(\overline{\sX_{-}})$ is defined as,
    \begin{align}\label{eq:DefPMinus}
    \big(\cP^{-}_{\ell}g^{-}\big)(s,i):=\bE_{s,i}\Big(g^{-}\Big(\tau_{\ell}^{-},X_{\tau_{\ell}^{-}}\Big)\Big),\quad (s,i)\in\overline{\sX_{-}}.
    \end{align}
%\item For any $(s,i)\in\overline{\sX_{+}}$, we define
%    \begin{align}\label{eq:DefGenGPlus}
%    \big(G^{+}g^{+}\big)(s,i):=\lim_{\ell\rightarrow 0+}\frac{1}{\ell}\big(\cP^{+}_{\ell}g^{+}(s,i)-g^{+}(s,i)\big),
%    \end{align}
%    for any $g^{+}\in C_{0}(\overline{\sX_{+}})$ such that the limit in \eqref{eq:DefGenGPlus} exists and is finite.
%\item For any $(s,i)\in\overline{\sX_{-}}$, we define
%    \begin{align}\label{eq:DefGenGMinus}
%    \big(G^{-}g^{-}\big)(s,i):=\lim_{\ell\rightarrow 0+}\frac{1}{\ell}\big(\cP^{-}_{\ell}g^{-}(s,i)-g^{-}(s,i)\big),
%    \end{align}
%    for all $g^{-}\in C_{0}(\overline{\sX_{+}})$ such that the above limit in \eqref{eq:DefGenGMinus} exists and is finite.
\end{itemize}

The proposition below follows from the strong Markov property of $(X_t)_{t\in[s,\infty]}$. We refer to \cite[Section 2.3]{BieCheCiaGon2020} for the proof.
\begin{proposition}\label{prop:ExpgpsimJPPlus}
For $g^{+}\in \cB_b(\overline{\sX_{+}})$, $\ell\in(0,\infty)$, and $(s,i)\in\sX_{-}$, we have
\begin{align}\label{eq:ExpgpsimJPPlus}
\bE_{s,i}\Big(g^{+}\Big(\tau_{\ell}^{+},X_{\tau_{\ell}^{+}}\Big)\Big)=\big(J^{+}\cP_{\ell}^{+}g^{+}\big)(s,i).
\end{align}
Analogously, for $g^-\in\cB_b(\overline{\sX_-})$, $\ell\in(0,\infty)$, and $(s,i)\in\sX_+$, we have
\begin{align*}
\bE_{s,i}\Big(g^{-}\Big(\tau_{\ell}^{-},X_{\tau_{\ell}^{-}}\Big)\Big)=\big(J^{-}\cP_{\ell}^{-}g^{-}\big)(s,i).
\end{align*}
\end{proposition}

\begin{remark}\label{rmk:WH}
Let $C_{0}^{1}(\overline{\sX})$ be the space of functions $f\in C_{0}(\overline{\sX})$ such that, for any $i\in\mathbf{E}$, $\partial f(\cdot,i)/\partial s$ exists and belongs to $C_{0}(\bR_{+})$. \cite{BieCheCiaGon2020} shows that, if we assume additionally that $s\mapsto\mathsf\Lambda_s$ is continuous, then $J^\pm$ and $\cP^\pm$ can be uniquely characterized by certain operator equation. More precisely: Define $\mV:=\textnormal{diag}\,\{v(i):i\in\mathbf{E}\}$, $\wt{\mathsf\Lambda}g(s,i):=[\Lambda_s g(s, \cdot)]_i$ and consider the following equation in unknown $(S^+,H^+,S^-,H^-)$
\begin{align}\label{eq:WH}
\mV^{-1}\bigg(\frac{\partial}{\partial s}+\wt{\mathsf{\Lambda}}\bigg) \begin{pmatrix} I^{+} & S^{-} \\ S^{+} & I^{-} \end{pmatrix} \begin{pmatrix} g^{+} \\ g^{-} \end{pmatrix} = \begin{pmatrix} I^{+} & S^{-} \\ S^{+} & I^{-} \end{pmatrix} \begin{pmatrix} H^{+} & \mathsf{0} \\ \mathsf{0} & -H^{-} \end{pmatrix} \begin{pmatrix} g^{+} \\ g^{-} \end{pmatrix},\quad g^{\pm}\in C_{0}^{1}(\overline{\sX_{\pm}}),
\end{align}
subject to the conditions below:
\begin{itemize}
\item [$(\textnormal{a}^{\pm})$] $S^{\pm}:C_{0}(\overline{\sX_{\pm}})\rightarrow C_{0}(\overline{\sX_{\mp}})$ is a bounded operator such that
    \begin{itemize}
    \item[(i)] for any $g^{\pm}\in C_{c}(\overline{\sX_{\pm}})$ with $\supp g^{\pm}\subset[0,\eta_{g^{\pm}}]\times\mathbf{E}_{\pm}$ for some constant $\eta_{g^{\pm}}\in(0,\infty)$, we have $\supp S^{\pm}g^{\pm}\subset[0,\eta_{g^{\pm}}]\times\mathbf{E}_{\mp}$;
    \item[(ii)] for any $g^{\pm}\in C_{0}^{1}(\overline{\sX_{\pm}})$, we have $S^{\pm}g^{\pm}\in C_{0}^{1}(\overline{\sX_{\mp}})$.
    \end{itemize}
\item [$(\textnormal{b}^{\pm})$] $H^{\pm}$ is the strong generator of a strongly continuous positive contraction semigroup $(\cQ_{\ell}^{\pm})_{\ell\in\bR_{+}}$ on $C_{0}(\overline{\sX_{\pm}})$ with domain $\sD(H^{\pm})=C_{0}^{1}(\overline{\sX_{\pm}})$.
\end{itemize}
Then, \eqref{eq:WH} has a unique solution. Moreover, restricted to $C_{0}(\overline{\sX_{\pm}})$, we have $J^\pm=S^\pm$. In addition, $(\cQ^\pm_\ell)_{\ell\in\bR^+}$ are strongly continuous positive contraction semigroups with generators $H^\pm$ and $\cP^\pm_\ell=Q^\pm_\ell$. %We refer to \cite{BarlowRogersWilliams1980}, \cite{KennedyWilliams1990}, \cite{Rogers1994}, \cite{JiangPistorius2008} for similar results in time-homogeneous setting.
\end{remark}

\section{Main Result}\label{sec:MainResults}

Note that $\phi_0(s)=0$. To facilitate the investigation of the exit time of $(\phi_t(s))_{t\ge s}$ from interval $[-\ell^-,\ell^+]$, where $\ell^\pm\ge 0$, we define
\begin{align*}
\xi^{+}_{\ell^-,\ell^+}(s):=\inf\set{t\in[s,\tau^{-}_{\ell^-}(s)): \phi_t(s)>\ell^+} = \1_{\big\{\tau^{+}_{\ell^+}(s) < \tau^{-}_{\ell^-}(s)\big\}} \tau^{+}_{\ell^+}(s) + \1_{\big\{\tau^{+}_{\ell^+}(s) \ge \tau^{-}_{\ell^-}(s)\big\}}\cdot\infty,
\end{align*}
and
\begin{align*}
\xi^{-}_{\ell^-,\ell^+}(s):=\inf\set{t\in[s,\tau^{+}_{\ell^+}(s)): \phi_t(s)<-\ell^-} = \1_{\big\{\tau^{-}_{\ell^-}(s) < \tau^{+}_{\ell^+}(s)\big\}} \tau^{-}_{\ell^-}(s) + \1_{\big\{\tau^{-}_{\ell^-}(s) \ge \tau^{+}_{\ell^+}(s)\big\}}\cdot\infty,
\end{align*}
where we adopted the usual convention that $\inf\emptyset=\infty$. Note that for a fixed $\omega$ at least one of $\xi^{+}_{\ell^-,\ell^+}(s,\omega)$ and $\xi^{-}_{\ell^-,\ell^+}(s, \omega)$ equals to $\infty$.  Clearly,
\begin{align}\label{eq:XxiRange}
X_{\xi^{+}_{\ell^-,\ell^+}(s)}\in\mathbf{E}_+\cup\set{\partial}\quad\text{{and}}\quad X_{\xi^{-}_{\ell^-,\ell^+}(s)}\in\mathbf{E}_-\cup\set{\partial}.
\end{align}
We define $\Xi^{+}_{\ell^-,\ell^+}:\cB_b(\overline{\sX_+})\to \cB_b(\overline\sX)$ as
\begin{align}\label{eq:DefXiplus}
{\Big(\Xi^{+}_{\ell^-,\ell^+}g^+\Big)(s,i)} := \bE_{s,i}\left(g^+\left(\xi^{+}_{\ell^-,\ell^+}(s),X_{\xi^{+}_{\ell^-,\ell^+}(s)}\right)\right),\quad {(s,i)\in\overline\sX,}
\end{align}
{where $g^+\in\cB_b(\overline{\sX_+})$.} Note that we consider $g^+\in \cB_b(\overline{\sX_+})$ instead of $g\in \cB_b(\overline{\sX})$ due to \eqref{eq:XxiRange}. %We also point out that $\Xi^+_{\ell^-,\ell^+}g^+(\infty,\cdot)=0$ as $\xi^+_{\ell^-,\ell^+}(\infty)=\infty$ and $g(\infty,\cdot)=0$.
Similarly, we define $\Xi^{-}_{\ell^-,\ell^+}:\cB_b(\overline{\sX_-})\to \cB_b(\overline\sX)$ as
\begin{align}\label{eq:DefXiminus}
{\Big(\Xi^{-}_{\ell^-,\ell^+}g^-\Big)(s,i)} := \bE_{s,i}\left(g^-\left(\xi^{-}_{\ell^-,\ell^+}(s), X_{\xi^{-}_{\ell^-,\ell^+}(s)}\right)\right),\quad {(s,i)\in\overline\sX,}
\end{align}
{where $g^-\in\cB_b(\overline{\sX_-})$.}

%We define the first exist time of $\phi$ from the interval $[a,b]$ as below
%\begin{align*}
%\xi_{a,b}(s) := \inf\set{t\ge s: \phi_t(s)\notin[a,b]}.
%\end{align*}
Note that events $\set{\xi^{+}_{\ell^-,\ell^+}(s)<\infty}$, {$\set{\xi^{-}_{\ell^-,\ell^+}(s)<\infty}$, and $\set{\xi^{+}_{\ell^-,\ell^+}(s)=\xi^{-}_{\ell^-,\ell^+}(s)=\infty}$} forms a partition for $\Omega$. It follows that
\begin{align}\label{eq:DecomptauMin}
\!\tau^{+}_{\ell^+}(s)\!\wedge\!\tau^{-}_{\ell^-}(s)\!=\!\1_{\set{\xi^{+}_{\ell^-,\ell^+}(s)<\infty}} \xi^{+}_{\ell^-,\ell^+}(s) \!+\! \1_{\set{\xi^{-}_{\ell^-,\ell^+}(s)<\infty}}\xi^{-}_{\ell^-,\ell^+}(s)\!+\!{\1_{\set{\xi^{+}_{\ell^-,\ell^+}(s)=\xi^{-}_{\ell^-,\ell^+}(s)=\infty}}}\!\cdot\!\infty.
\end{align}
For any $g\in \cB_b(\overline\sX)$, if we let $g^\pm\in \cB_b(\overline{\sX_\pm})$ satisfy {$g^\pm(s,i)=g(s,i)$ for $(s,i)\in\sX_\pm$}, then by \eqref{eq:XxiRange} and \eqref{eq:DecomptauMin}, {we have}
\begin{align}\label{eq:XipmSum}
&\bE_{s,i}\left(g\left(\tau^{-}_{\ell^-}\wedge\tau^{+}_{\ell^+}, X_{\tau^{-}_{\ell^-}\wedge\tau^{+}_{\ell^+}}\right)\right) = {\left(\Xi^{+}_{\ell^-,\ell^+}g^+\right)(s,i) + \left(\Xi^{-}_{\ell^-,\ell^+}g^-\right)(s,i)}.
\end{align}
Recall that the left hand side of \eqref{eq:XipmSum} is the expectation showing in \eqref{eq:MainProblem}. We therefore focus on computing $\Xi^{+}_{\ell^-,\ell^+}$ and $\Xi^{-}_{\ell^-,\ell^+}$.

%The following lemma regards the convergence of $\sum_{k=0}^\infty (J^-\cP^-_{b-a}J^+\cP^+_{b-a})^k g^+$ for $g^+$ that decays exponentially fast. It also indicates certain invertibility of $I^+-J^-\cP^-_{b-a}J^+\cP^+_{b-a}$.
In what follows we say that $g^\pm\in\cB_b(\overline{\sX_{\pm}})$ decays exponentially fast to zero if there are constants $C,c>0$ such that $\max_{i\in\mathbf{E}_\pm}|g^\pm(s,i)| \le C e^{-cs}$ for $s\in\bR_+$. {The proofs of results below are presented in Section \ref{sec:Proofs}.} We begin with the following lemma.
\begin{lemma}\label{lem:NeumannSum}
Suppose {that} $g^+\in \cB_b(\overline{\sX_+})$ decays exponentially fast to zero. {Then for any $\ell^+,\ell^-\in\bR_+$,
\begin{align}\label{eq:SeriesgPlus}
\sum_{n=1}^{\infty}\left(J^{-}\cP^-_{\ell^{-}+\ell^{+}}J^{+}\cP^{+}_{\ell^{-}+\ell^{+}}\right)^n g^+
\end{align}
converges in $\|\cdot\|_\infty$}. Analogously, suppose {that} $g^-\in \cB_b(\overline{\sX_-})$ decays exponentially fast to zero. {Then for any $\ell^+,\ell^-\in\bR_+$,
\begin{align}\label{eq:SeriesgMinus}
\sum_{n=1}^{\infty}\left(J^{+}\cP^+_{\ell^{-}+\ell^{+}}J^{-}\cP^{-}_{\ell^{-}+\ell^{+}}\right)^n g^-
\end{align}
converges in $\|\cdot\|_\infty$.}
%and
%\begin{align}\label{eq:NeumannSum}
%g^+ &=(I^+-J^-\cP^-_{\ell^{-}+\ell^{+}}J^+\cP^+_{\ell^{-}+\ell^{+}})\sum_{k=0}^\infty (J^-\cP^-_{\ell^{-}+\ell^{+}}J^+\cP^+_{\ell^{-}+\ell^{+}})^n g^+, \quad \ell^\pm\in\bR_+.%\nonumber\\
%& = \sum_{k=0}^\infty (J^-\cP^-_{b-a}J^+\cP^+_{b-a})^k(I^+-J^-\cP^-_{b-a}J^+\cP^+_{b-a}) g^+.
%\end{align}
\end{lemma}

The ensuing theorem is our main result. %, which express $\Xi^\pm$ in terms of operators defined in \eqref{eq:DefJPlus}-\eqref{eq:DefPMinus}.
\begin{theorem}\label{thm:TwoSidedExit}
For $g^\pm\in \cB_b(\overline{\sX_\pm})$ decaying exponentially fast to zero, {for any $\ell^+,\ell^-\in\bR_+$} and $(s,i)\in\sX$, we have
\begin{align}\label{eq:Psiplus}
{\left(\Xi^{+}_{\ell^-,\ell^+}g^+\right)\!(s,i)} \!&=\! \!\left[\!\left(\!\begin{pmatrix}I^+\\ J^+\end{pmatrix}\!\cP^{+}_{\ell^+}\!-\!\begin{pmatrix}J^-\\ I^-\end{pmatrix}\!\cP^{-}_{\ell^-} J^+\cP^+_{\ell^{-}+\ell^{+}}\!\right)\!\sum_{n=0}^\infty (J^-\cP^-_{\ell^{-}+\ell^{+}}J^+\cP^+_{\ell^{-}+\ell^{+}})^n g^+\right]\!(s,i),\\
\label{eq:Psiminus}
{\left(\Xi^{-}_{\ell^-,\ell^+}g^-\right)\!(s,i)} \!&=\!\! \left[\!\left(\!\begin{pmatrix}J^-\\ I^-\end{pmatrix}\!\cP^{-}_{\ell^-}  \!-\! \begin{pmatrix}J^+\\ I^+\end{pmatrix}\!\cP^{+}_{\ell^+} J^-\cP^-_{\ell^{+}+\ell^{-}}\!\right)\!\sum_{n=0}^\infty (J^+\cP^+_{\ell^{-}+\ell^{+}}J^-\cP^-_{\ell^{-}+\ell^{+}})^n g^-\right]\!(s,i).
\end{align}
\end{theorem}

\begin{remark}\label{rmk:Invertibility}
If one of $J^+$, $J^-$, $\cP^+_{\ell^-+\ell^+}$ or $\cP^-_{\ell^-+\ell^+}$ has an operator norm $\|\cdot\|_\infty$ that is strictly less than 1, then $\sum_{n=0}^\infty (J^-\cP^-_{\ell^{-}+\ell^{+}}J^+\cP^+_{\ell^{-}+\ell^{+}})^n$ converges in operator norm $\|\cdot\|_\infty$ to $(I^+-J^-\cP^-_{\ell^{-}+\ell^{+}}J^+\cP^+_{\ell^{-}+\ell^{+}})^{-1}$, and further allows removing the condition that $g^\pm$ decays exponentially fast to zero from the statement of Theorem \ref{thm:TwoSidedExit}. The proof of this modification of Theorem \ref{thm:TwoSidedExit} is similar to the proof in Section \ref{subsec:ProofThm}, and will be omitted. A sufficient condition for $\|\cP^\pm_{\ell^-+\ell^+}\|_\infty < 1$ is that there is a $c>0$ such that each row of $\mathsf{\Lambda}_s$ sums up to a number less than $-c$ for all $s\in\bR_+$. Indeed, this introduces an exponential killing at the rate of at least $c$ to $(X_t)_{t\in[s,\infty]}$, for all $s\in\bR_+$. Consequently, $\|\cP^\pm_{\ell^-+\ell^+}\|_\infty \le 1- \exp(-c(\ell^-+\ell^+)\|v\|_\infty^{-1})$, {since} for $\ell>0$ and $s\in\bR_+$ we have $\tau^\pm_{\ell}(s)-s\ge \ell\|v\|_\infty^{-1}$.
\end{remark}

\begin{remark}
In this remark we relate \eqref{eq:Psiplus} to \cite[Proposition 1]{JiangPistorius2008} without additive Brownian noise by letting the Markov family $\cM$ be time-homogeneous, namely, for all $s\in\bR_{+}$,  $\mathsf{\Lambda}_{s}=\mathsf{\Lambda}_0=:\mathsf{\Lambda}$.
For simplicity, we assume that there is a $c>0$ such that each row of $\mathsf{\Lambda}$ sums up to a number smaller than $-c$.
We define $\zeta(s):=\inf\set{t\ge s: X_t=\partial}$; we will omit $s$ in $\zeta(s)$ when no confusion arises. In this case, observe that {
\begin{align*}
\bP_{s,i}\Big(X_{\tau^\pm_\ell(s)} = k\Big) = \bP_{s,i}\Big(\tau^\pm_\ell(s)<\zeta(s), X_{\tau^\pm_\ell(s)} = k\Big) = \bP_{0,i}\Big(\tau^\pm_\ell(0)<\zeta(0), X_{\tau^\pm_\ell(0)} = k\Big)
\end{align*}
for any $(s,i)\in\sX$, $k\in\mathbf E_\pm$,} and $\ell\in\bR_+$. Moreover, it is shown in \cite[Theorem 2, (24)]{JiangPistorius2008} (see also \cite{BarlowRogersWilliams1980}, \cite[Remark 3.5]{BieCheCiaGon2020}) that there are sub-Markovian matrices $\mQ^+$ and $\mQ^-$ on $\mathbf E_+$ and $\mathbf E_-$, respectively, such that
\begin{gather*}
{\bP_{0,i}\Big(\tau^+_\ell<\zeta, X_{\tau^+_\ell} = k\Big) = \big[e^{\ell\mQ^+} \big]_{ik}, \quad (i,k)\in\mathbf E_+^2,\quad\ell\in\bR_+,}\\
{\bP_{0,i}\Big(\tau^-_\ell<\zeta, X_{\tau^-_\ell} = k\Big) = \big[e^{\ell\mQ^-} \big]_{ik}, \quad (i,k)\in\mathbf E_-^2,\quad\ell\in\bR_+.}
\end{gather*}
We define $\mJ^+\in\bR^{|\mathbf{E}_-|\times|\mathbf{E}_+|}$ and $\mJ^-\in\bR^{|\mathbf{E}_+|\times|\mathbf{E}_-|}$ as
\begin{gather*}
\mathsf J^+_{ik} := {\bP_{0,i}\Big(\tau^+_0<\zeta, X_{\tau^+_0} = k\Big)},\quad (i,k)\in\mathbf E_-\times\mathbf E_+,\\
\mathsf J^-_{ik} := {\bP_{0,i}\Big(\tau^-_0<\zeta, X_{\tau^-_0} = k\Big)},\quad (i,k)\in\mathbf E_+\times\mathbf E_-.
\end{gather*}
Then, for any $g^+$ satisfying $g^+(s,i) = g^+(0,i)$ for all $(s,i)\in {\sX}_+$, setting $\mathsf g^+(\cdot):=g^+(0,\cdot)$, we have
\begin{gather*}
{\big(J^+g^+\big)(s,i)} = \sum_{k\in\mathbf{E}_+}{\bE_{s,i}\bigg(\1_{\set{\tau^+_0<\zeta, X_{\tau^+_0}=k}} \mathsf g^+(k)\bigg)} = [\mJ^+ \mathsf g^+]_i,\quad (s,i)\in\sX^-,\\
{\big(\cP^+_\ell g^+\big)(s,i)} = \sum_{k\in\mathbf{E}_+}{\bE_{s,i}\bigg(\1_{\set{\tau^+_\ell<\zeta, X_{\tau^+_\ell}=k}} \mathsf g^+(k)\bigg) = [e^{\ell\mQ^+} \mathsf g^+]_i,\quad (s,i)\in\sX^+,\quad\ell\ge 0.}
\end{gather*}
In view of the exponential killing, with similar reasoning as in Remark \ref{rmk:Invertibility}, we have
%\begin{align*}
%\left|\bE_{s,i}\Big(g^{+}\Big(\tau_{\ell}^{+},X_{\tau_{\ell}^{+}}\Big)\Big)\right| \le \|g^+\|_\infty \bP(\tau_{\ell}^{+}<\infty) < \|g^+\|_\infty, \quad (s,i)\in\sX,\,\ell> 0,
%\end{align*}
%i.e.,
$\|e^{(\ell^++\ell^-)\mQ^\pm}\|_\infty\le e^{-c(\ell^++\ell^-)\|v\|_\infty^{-1}}$. It follows that
\begin{align*}
\sum_{n=0}^\infty {\left(\Big(J^-\cP^-_{\ell^{-}+\ell^{+}}J^+\cP^+_{\ell^{-}+\ell^{+}}\Big)^ng^+\right)(0,\cdot)} &= \sum_{n=0}^\infty \left(\mJ^-e^{(\ell^-+\ell^+)\mQ^-}\mJ^+e^{(\ell^-+\ell^+)\mQ^+}\right)^n \mathsf g^+ \\
&= \left(\mI-\mJ^-e^{(\ell^-+\ell^+)\mQ^-}\mJ^+e^{(\ell^-+\ell^+)\mQ^+}\right)^{-1} \mathsf g^+.
\end{align*}
Consequently, for any $(s,i)\in\sX$, the right hand side of \eqref{eq:Psiplus} becomes
\begin{align*}
\left[\left(\begin{pmatrix}\mI^+\\ \mJ^+\end{pmatrix}e^{\ell^+\mQ^+}  - \begin{pmatrix}\mJ^-\\ \mI^-\end{pmatrix}e^{\ell^-\mQ^-} \mJ^+e^{(\ell^-+\ell^+)\mQ^+}\right)\left(\mI-\mJ^-e^{(\ell^-+\ell^+)\mQ^-}\mJ^+e^{(\ell^-+\ell^+)\mQ^+}\right)^{-1} \mathsf g^+\right]_i,
\end{align*}
which shows that \eqref{eq:Psiplus} as a special case of \cite[Proposition 1, equation (31)]{JiangPistorius2008} without additive Brownian noise. For \eqref{eq:Psiminus}, the sanity check can be carried out analogously.
\end{remark}

The knowledge of $\Xi^\pm_{\ell^-,\ell^+}$ can also be used to study the law of $X_T$ before the exit time, that is the law of $X_T$ restricted to the set $\tau^{-}_{\ell^-}\wedge\tau^{+}_{\ell^+}> T$. Observe that
\begin{align}
\bE_{s,i}\left(h(X_T)\1_{\left\{\tau^{-}_{\ell^-}\wedge\tau^{+}_{\ell^+}> T\right\}}\right) = \bE_{s,i}(h(X_T)) - \bE_{s,i}\left(h(X_T)\1_{\left\{\tau^{-}_{\ell^-}\wedge\tau^{+}_{\ell^+}\le T\right\}}\right),
\end{align}
it is sufficient to investigate the second term of the right hand side. The proposition below represents the desired quantity in terms of $\Xi^\pm$.
\begin{proposition}\label{prop:AfterExit}
For any $h:\mathbf E\to\bR$ and $\ell^\pm\in\bR_+$, we have
\begin{align*}
\bE_{s,i}\left(h(X_T)\1_{\left\{\tau^{-}_{\ell^-}\wedge\tau^{+}_{\ell^+}\le T\right\}}\right) = {\Big(\Xi^{+}_{\ell^-,\ell^+}k^+\Big)(s,i)+\Big(\Xi^{-}_{\ell^-,\ell^+}k^-\Big)(s,i)},
\end{align*}
where
\begin{align}\label{eq:SpecialDefgplus}
k^\pm(t,j) := \1_{[0,T]}(t)\,{\bE_{t,j}\big(h(X_{T})\big)},\quad (t,j)\in\sX_\pm.
\end{align}
\end{proposition}

\section{Proofs}\label{sec:Proofs}

\subsection{Proof of Lemma \ref{lem:NeumannSum}}\label{subsec:ProofNeumannSum}
The proof of Lemma \ref{lem:NeumannSum} relies on the following inequality.
\begin{lemma}\label{lem:EstJPJPk}
Suppose {that} $g^+(t,j)=\1_{[0,T]}(t)$ for some $T\in\bR_+$. Then, for any $(s,i)\in\sX_+$ and {$\ell^\pm\in\bR_+$,} we have
\begin{align*}
{\Big(J^-\cP^-_{\ell^{-}+\ell^{+}}J^+\cP^+_{\ell^{-}+\ell^{+}} g^+\Big)(s,i)} \le \1_{[0,T]}(s)\;\left(1 - e^{-K(T-s)}\right)^{2}.
\end{align*}
\end{lemma}
\begin{proof}
To start with, we let $\gamma(s):=\inf\set{t\ge s: X_t\neq X_s}$. Clearly, $\bP_{s,i}(\gamma(s)\le T)=0$ for $s>T$.
It is known that (cf. \cite[Section 8.4.2, (8.4.21)]{RolskiSchmidliSchmidtTeugels1999})
\begin{align}\label{eq:ProbFirstJump}
{\bP_{s,i}\big(\gamma(s)\le T\big)} = 1-\exp\left(-\int_{s}^T\mathsf{\Lambda}_u(i,i)du\right) \le 1 - e^{-K(T-s)},\quad s\le T,
\end{align}
where we {used} Assumption \ref{assump:GenLambda}-(i) for the inequality. It follows from Lemma \ref{lem:RangeXTaupm} and \eqref{eq:ProbFirstJump} that
\begin{align}\label{eq:Probtau0Est}
{\bP_{s,i}\big(\tau_0^{\pm}(s)\le T\big) \le \bP_{s,i}\big(\gamma(s)\le T\big)} \le 1 - e^{-K(T-s)},\quad s\le T, i\in\mathbf{E}_\mp.
\end{align}
Next, observe that $\tau^\pm_\ell(s)\ge s$. By \eqref{eq:DefPPlus}, we have
\begin{align*}
{\Big(\cP^+_{\ell^{-}+\ell^{+}}g^+\Big)(s,i) = \bP_{s,i}\big(\tau^+_{\ell^-+\ell^+} \le T\big)} \le g^+(s,i)
\end{align*}
Then, for $(s,i)\in\sX_-$, by \eqref{eq:DefJPlus} and \eqref{eq:Probtau0Est},  we have
\begin{align}\label{eq:EstJplusPplusgplus}
0 \le {\Big(J^+\cP^+_{\ell^{-}+\ell^{+}}g^+\Big)(s,i) \le \big(J^+g^+\big)(s,i) = \bP_{s,i}\big(\tau_0^{\pm}(s)\le T\big) \le \1_{[0,T]}(s)\Big(1 - e^{-K(T-s)}\Big)}.
\end{align}
%which implies that $\supp J^+\cP^+_{\ell^{-}+\ell^{+}}g^+\subset[0,T]\times\mathbf E_-$.
It follows from \eqref{eq:DefPMinus} and \eqref{eq:EstJplusPplusgplus} that, for $(s,i)\in\sX_-$,
\begin{align}\label{eq:EstJminusPminusJplusPplus}
0 \le {\Big(\cP^-_{\ell^{-}+\ell^{+}}J^+\cP^+_{\ell^{-}+\ell^{+}}g^+\Big)(s,i)} \le \sup_{r\ge s}\1_{[0,T]}(r){\Big(1 - e^{-K(T-r)}\Big)} \le {\1_{[0,T]}(s)\left(1 - e^{-K(T-s)}\right)}.
\end{align}
%which implies that $\supp\cP^-_{\ell^{-}+\ell^{+}}J^+\cP^+_{\ell^{-}+\ell^{+}}g^+\subset[0,T]\times\mathbf{E}_-$.
Finally, by \eqref{eq:DefJMinus}, \eqref{eq:EstJminusPminusJplusPplus} and \eqref{eq:Probtau0Est},
\begin{align*}
{\Big(J^-\cP^-_{\ell^{-}+\ell^{+}}J^+\cP^+_{\ell^{-}+\ell^{+}}g^+\Big)(s,i)} & \le {\bE_{s,i}\left(\1_{[0,T]}(\tau_{0}^{-})\sup_{r\ge s,\,i\in\mathbf{E}_-}\Big(\cP^-_{\ell^{-}+\ell^{+}}J^+\cP^+_{\ell^{-}+\ell^{+}}g^+\Big)(r,i)\right)}\\
& \le {\1_{[0,T]}(s)\left(1 - e^{-K(T-s)}\right)^2},
\end{align*}
which completes the proof.
%Moreover, the support of $J^-\cP^-_{b-a}J^+\cP^+_{b-a}g^+$ is still included by $[0,T]\times\mathbf{E}$ due the a similar argument as before. By induction, we conclude that
%\begin{align*}
%(J^-\cP^-_{\ell^{-}+\ell^{+}}J^+\cP^+_{\ell^{-}+\ell^{+}})^n g^+(s,i) \le \1_{[0,T]}(s)\;\left(1 - e^{-K(T-s)}\right)^{2n}.
%\end{align*}
%Finally, for $g^+$ with  $\cup_{i\in\sX_+}\supp g^+(\cdot,i)\subseteq[0,T]$,
%\begin{align*}
%(J^-\cP^-_{\ell^{-}+\ell^{+}}J^+\cP^+_{\ell^{-}+\ell^{+}})^n g^+(s,i) \le \|g^+\|_\infty \1_{[0,T]}(s)\;\left(1 - e^{-K(T-s)}\right)^{2n},
%\end{align*}
%which completes the proof.
\end{proof}

\begin{proof}[Proof of Lemma \ref{lem:NeumannSum}]
{We will only present the proof for the convergence of \eqref{eq:SeriesgPlus} as the proof for the convergence of \eqref{eq:SeriesgMinus} follows analogously.} Without loss of generality, we assume {that} $g^+(t,i) = e^{-ct}$ for some $c>0$ and define $g^{+}_{k}(s,i):={\sum_{j=1}^{2^k-1}\1_{[0,T_{j}^{k}]}(s)/2^{k}}$, where $T_{j}^k := \inf\{t\ge 0:e^{-ct} \le 1-j2^{-k}\} = c^{-1} {\ln(2^k/(2^k-j))}$ for $j=1,\dots,2^k-1$. Therefore, by Lemma \ref{lem:EstJPJPk}, {
\begin{align*}
\left(\!\Big(J^{-}\cP^-_{\ell^{-}+\ell^{+}}J^{+}\cP^{+}_{\ell^{-}+\ell^{+}}\!\Big)g^+_k\right)\!(0,i)\leq\frac{1}{2^k}\!\sum_{j=1}^{2^k-1}\!\left(\!1\!-\!\bigg(\!1\!-\!\frac{j}{2^k}\!\bigg)^{K/c}\right)^{2} \!\!\le\!\int_{0}^{1}\!\big(1\!-\!(1\!-\!x)^{K/c}\big)^{2}\!\dif x =: C_{K,c} < 1.
\end{align*}
Since} $g^{+}_{k}$ is non-decreasing in $k\in\bN$ and $\lim_{k\to\infty}g^{+}_{k}(s,i) = g^{+}(s,i)$ for $(s,i)\in\sX_{+}$, by \eqref{eq:DefPPlus} and monotone convergence, we have {that} $\cP^+_{\ell^{-}+\ell^{+}} g^+_k$ is non-decreasing {in $k\in\bN$ and that
\begin{align*}
\lim_{k\to\infty}\Big(\cP^{+}_{\ell^{-}+\ell^{+}}g^{+}_{k}\Big)(s,i)=\Big(\cP^{+}_{\ell^{-}+\ell^{+}}g^{+}\Big)(s,i),\quad (s,i)\in\sX_{+}.
\end{align*}
It follows from \eqref{eq:DefJPlus} and monotone convergence that
\begin{align*}
\lim_{k\to\infty}\Big(J^+\cP^{+}_{\ell^{-}+\ell^{+}}g^{+}_{k}\Big)(s,i) = \Big(J^+\cP^{+}_{\ell^{-}+\ell^{+}}g^{+}\Big)(s,i),\quad (s,i)\in\sX_{-}.
\end{align*}
By a similar reasoning we deduce that}
\begin{align*}
{\left(\Big(J^{-}\cP^-_{\ell^{-}+\ell^{+}}J^{+}\cP^{+}_{\ell^{-}+\ell^{+}}\Big)g^+\right)(0,i)}  = \lim_{k\to\infty} {\left(\Big(J^{-}\cP^-_{\ell^{-}+\ell^{+}}J^{+}\cP^{+}_{\ell^{-}+\ell^{+}}\Big)g^+_k\right)(0,i)} \le  C_{K,c}.
\end{align*}
Note that {
\begin{align*}
\left(\Big(J^{-}\cP^-_{\ell^{-}+\ell^{+}}J^{+}\cP^{+}_{\ell^{-}+\ell^{+}}\Big)g^+\right)(s,i) = e^{-cs}\left(\Big(J^{-}\cP^-_{\ell^{-}+\ell^{+}}J^{+}\cP^{+}_{\ell^{-}+\ell^{+}}\Big)g^{+,s}\right)(s,i),
\end{align*}
where $g^{+,s}(t,j):= e^{-c(t-s)}$} for $t\ge s$ and $j\in\mathbb E$. {A similar argument as before shows that}
\begin{align*}
{\left(\Big(J^{-}\cP^-_{\ell^{-}+\ell^{+}}J^{+}\cP^{+}_{\ell^{-}+\ell^{+}}\Big)g^+\right)(s,i)} \le C_{K,c}\,e^{-cs} = C_{K,c}\,g^+(s,i).
\end{align*}
Invoking the linearity and {nonnegativity} of $J^-\cP^-_{\ell^{-}+\ell^{+}}J^+\cP^+_{\ell^{-}+\ell^{+}}$, by iteration, we {have
\begin{align*}
\left(\Big(J^{-}\cP^-_{\ell^{-}+\ell^{+}}J^{+}\cP^{+}_{\ell^{-}+\ell^{+}}\Big)^ng^+\right)(s,i) \le C_{K,c}^{n}\,g^+(s,i),\quad\text{for any }\,n\in\bN.
\end{align*}
%Note that the left hand side above is nonnegative bounded by $C_{K,c}^{n}$ for all $(s,i)\in\overline{\sX}_{+}$ and $n\in\bN$.
Finally, we obtain that}
\begin{align*}
\left\|\sum_{n=1}^{\infty}{\Big(J^{-}\cP^-_{\ell^{-}+\ell^{+}}J^{+}\cP^{+}_{\ell^{-}+\ell^{+}}\Big)^n g^+}\right\|_\infty \le \sum_{n=1}^{\infty}\left\| {\Big(J^{-}\cP^-_{\ell^{-}+\ell^{+}}J^{+}\cP^{+}_{\ell^{-}+\ell^{+}}\Big)^n g^+}\right\|_\infty \le \sum_{n=1}^\infty C^n_{K,c}\,\|g^+\|_\infty < \infty,
\end{align*}
from which the strong convergence follows immediately.
\end{proof}

\subsection{Auxiliary Markov families}\label{subsec:TimeHomogen}
In this subsection, we introduce an  auxiliary time-inhomogenous Markov family $\wh{\cM}$ and an auxiliary time-homogenous Markov family $\widetilde \cM$. Most of the technical details presented in this section, except for those involving two-sided exit times, are similar to \cite[Section 4.1]{BieCheCiaGon2020}.

We start by introducing some more notations of spaces and $\sigma$-fields. Let $\sY:=\overline{\mathbf{E}}\times\bR$, and the Borel $\sigma$-field on $\sY$ is denoted by $\cB(\sY):=2^{\mathbf{E}}\otimes\cB(\bR)$. Accordingly, let $\overline{\sY}:=\sY\cup\{(\partial,\infty)\}$ be the one-point completion of $\sY$, and $\cB(\overline{\sY}):=\sigma(\cB(\sY)\cup\set{(\partial,\infty)})$. Moreover, we set $\sZ:=\bR_{+}\times\sY=\sX\times\bR$ and $\overline{\sZ}:=\sZ\cup\{(\infty,\partial,\infty)\}$.

Let $\wh{\Omega}$ be the set of c\`{a}dl\`{a}g functions $\wh{\omega}$ on $\bR_{+}$ taking values in $\sY$. We define $\wh{\omega}(\infty):=(\partial,\infty)$ for every $\wh{\omega}\in\wh{\Omega}$. As shown in \cite[Appendix A]{BieCheCiaGon2020}, one can construct a {\it standard} canonical time-inhomogeneous Markov family (cf. \cite[Definition I.6.6]{GikhmanSkorokhod2004})
\begin{align*}
\wh{\cM}:=\big\{\big(\wh{\Omega},\wh{\sF},\wh{\bF}_{s},(\wh{X}_{t},\wh{\varphi}_{t})_{t\in[s,\infty]},\wh{\bP}_{s,(i,a)}\big),\,(s,i,a)\in\overline{\sZ}\big\}
\end{align*}
with transition function $\widehat P$ given by
\begin{align}\label{eq:DefTranProbXvarphi}
\wh{P}(s,(i,a),t,A):=\bP_{s,i}\bigg(\bigg(X_{t},\,a+\int_{s}^{t}v\big(X_{u}\big)\,du\bigg)\in A\bigg),
\end{align}
where $(s,i,a)\in\overline{\sZ}$, $t\in[s,\infty]$, and $A\in\cB(\overline{\sY})$. Note above $\widehat{X}_t$ takes values in $\overline{\mathbf E}$ and $\widehat{\varphi}_t$ takes values in $\bR$.

The lemma below reveals the probabilistic relationship between $\wh{\cM}$ and $\cM$. The proof can be found in \cite[Lemma A3]{BieCheCiaGon2020}.
\begin{lemma}\label{lem:XvarphiPathst}
For any $0\leq s\leq t<\infty$, let $\wh{\Omega}_{s,t}$ be the collection of all c\`{a}dl\`{a}g functions on $[s,t]$ taking values in $\sY$. Let $\wh{\sG}_{s,t}$ be the cylindrical $\sigma$-field on $\wh{\Omega}_{s,t}$ generated by $(\wh{X}_{u},\wh{\varphi}_{u})_{u\in[s,t]}$. Then, for any $a\in\bR$ and $C\in\wh{\sG}_{s,t}$,
\begin{align}\label{eq:XvarphiPathst}
\wh{\bP}_{s,(i,a)}\big(\big(\wh{X}_{r},\wh{\varphi}_{r}\big)_{r\in[s,t]}\in C\big)=\bP_{s,i}\bigg(\bigg(X_{r},\,a+\int_{s}^{r}v\big(X_{u}\big)\,du\bigg)_{r\in[s,t]}\in C\bigg).
\end{align}
\end{lemma}

As an immediate consequence of Lemma \ref{lem:XvarphiPathst}, $\widehat\cM$ has the following properties:
\begin{itemize}
\item[(i)] for any $(s,i,a)\in\overline{\sZ}$,
    \begin{align}\label{eq:LawXXStar}
    \text{the law of }\,\wh{X}\,\,\text{under }\,\wh{\bP}_{s,(i,a)}\,\,=\,\,\text{the law of }\,X\,\,\text{under }\,\bP_{s,i};
    \end{align}
\item[(ii)] for any $(s,i,a)\in\sZ$,
    \begin{align}\label{eq:DistvarphiInt}
    \wh{\bP}_{s,(i,a)}\bigg(\wh{\varphi}_{t}=a+\int_{s}^{t}v(\wh{X}_{u})\,du,\,\,\,\,\text{for all $t\in[s,\infty)$}\bigg)=1.
    \end{align}
\end{itemize}

Considering the standard Markov family $\wh{\cM}$, for any $s,\ell\in\bR$ with $-\ell^-\le\ell^+$, we define the following auxiliary first passage time
\begin{align*}
\wh{\tau}_{\ell}^{+}(s):=\inf\big\{t\in[s,\infty):\,\wh{\varphi}_{t}>\ell\big\},\quad\wh{\tau}_{\ell}^{-}(s):=\inf\big\{t\in[s,\infty):\,\wh{\varphi}_{t}<-\ell\big\}.
\end{align*}
For any $s,\ell^\pm\in\bR$ with $-\ell^-\le\ell^+$, we also define the two-sided exit time
\begin{align*}
\wh{\xi}_{\ell^-,\ell^+}^{+}(s):=\inf\big\{t\in{\big[s,\wh{\tau}^{-}_{\ell^-}(s)\big)}:\wh{\varphi}_{t}>\ell^+\big\},\quad \wh{\xi}_{\ell^-,\ell^+}^{-}(s):=\inf\big\{t\in{\big[s,\wh{\tau}^{+}_{\ell^+}(s)\big)}:\wh{\varphi}_{t}<-\ell^-\big\},
\end{align*}
The above are $\wh{\bF}_{s}$-stopping times in light of the continuity of $\wh{\varphi}$ and the right-continuity of the filtration $\wh{\bF}_{s}$. If no confusion arises, we will omit the $s$ in $\wh{\tau}^{\pm}_{\ell}(s)$ and $\wh{\xi}_{\ell^-,\ell^+}^{\pm}(s)$.

By \eqref{eq:DistvarphiInt} and \eqref{eq:LawXXStar}, for $a\le\ell$, we have
\begin{align}\label{eq:ExpPStarExpP}
\wh{\bE}_{s,(i,a)}\Big(g^{+}\Big(\wh{\tau}_{\ell}^{+},\wh{X}_{\wh{\tau}_{\ell}^{+}}\Big)\!\Big)&=\wh{\bE}_{s,(i,a)}\bigg(g^{+}\bigg(\!\inf\bigg\{u\!\geq\!s:a+\!\!\int_{s}^{u}\!v(\wh{X}_{r})dr\!>\!\ell\bigg\},\wh{X}_{\inf\big\{u\geq s:\,a+\int_{s}^{u}v(\wh{X}_{r})dr>\ell\big\}}\bigg)\!\bigg)\nonumber\\
&=\bE_{s,i}\bigg(g^{+}\bigg(\!\inf\bigg\{u\geq s:\!\int_{s}^{u}\!v\big(X_{r}\big)dr>\ell-a\bigg\},X_{\inf\big\{u\geq s:\int_{s}^{u}v(X_{r})dr>\ell-a\big\}}\bigg)\!\bigg)\nonumber\\
&=\bE_{s,i}\bigg(g^{+}\bigg(\tau_{\ell-a}^{+},\,X_{\tau_{\ell-a}^{+}}\bigg)\bigg).
\end{align}
%Therefore, for $a\le\ell$,
%\begin{align}
%\wh{\bE}_{s,(i,a)}\Big(g^{+}\Big(\wh{\tau}_{\ell}^{+},\wh{X}_{\wh{\tau}_{\ell}^{+}}\Big)\Big)=
%\bE_{s,i}\Big(g^{+}\Big(\tau_{\ell-a}^{+},X_{\tau_{\ell-a}^{+}}\Big)\Big).
%\end{align}
Similarly, for $a,\ell^\pm\in\bR$ satisfying $-\ell^-\le a \le \ell^+$,
\begin{align}\label{eq:ExpPStarExpPXi}
\wh{\bE}_{s,(i,a)}\Big(g^{+}\Big(\wh{\xi}_{\ell^-,\ell^+}^{+},\wh{X}_{\wh{\xi}_{\ell^-,\ell^+}^{+}}\Big)\Big)=\bE_{s,i}\Big(g^{+}\Big(\xi_{\ell^-+a,\ell^+-a}^{+},X_{\xi_{\ell^-+a,\ell^+-a}^{+}}\Big)\Big).
\end{align}

%\begin{proposition}\label{prop:ExpPStarExpP}
%For any $g^{+}\in \cB_b(\overline{\sX_{+}})$, $(s,i,a)\in\sZ$, and $\ell\in[a,\infty)$,
%\begin{align*}
%\bE_{s,(i,a)}\Big(g^{+}\Big(\tau_{\ell}^{+},X_{\tau_{\ell}^{+}}\Big)\Big)=\bE_{s,i}\Big(g^{+}\Big(\tau_{\ell-a}^{+},X_{\tau_{\ell-a}^{+}}\Big)\Big).
%\end{align*}
%\end{proposition}

%\begin{proof}
%By \eqref{eq:DistvarphiInt} and Lemma \ref{lem:XvarphiPathst}, we have
%\begin{align*}
%\bE_{s,(i,a)}\Big(g^{+}\Big(\tau_{\ell}^{+},X_{\tau_{\ell}^{+}}\Big)\!\Big)&=\bE_{s,(i,a)}\bigg(g^{+}\bigg(\!\inf\bigg\{u\!\geq\!s:a+\!\!\int_{s}^{u}\!v(X_{r})dr\!>\!\ell\bigg\},X_{\inf\big\{u\geq s:\,a+\int_{s}^{u}v(X_{r})dr>\ell\big\}}\bigg)\!\bigg)\\
%&=\bE_{s,i}\bigg(g^{+}\bigg(\!\inf\bigg\{u\geq s:\!\int_{s}^{u}\!v\big(X_{r}\big)dr>\ell-a\bigg\},X_{\inf\big\{u\geq s:\int_{s}^{u}v(X_{r})dr>\ell-a\big\}}\bigg)\!\bigg)\\
%&=\bE_{s,i}\bigg(g^{+}\bigg(\tau_{\ell-a}^{+},\,X_{\tau_{\ell-a}^{+}}\bigg)\bigg),
%\end{align*}
%which completes the proof.
%\end{proof}

Equalities \eqref{eq:ExpPStarExpP} and \eqref{eq:ExpPStarExpPXi} provide useful representations of the expectation under $\bP$. We will need still another representation of this expectation. Towards this end, we will first transform the time-inhomogeneous Markov family $\cM$ into a {\it time-homogeneous} Markov family
\begin{align*}
\wt{\cM}=\big\{\big(\wt{\Omega},\wt{\sF},\wt{\bF},({Z}_t)_{t\in\overline{\bR}_{+}},(\theta_{r})_{r\in\bR_{+}},\wt{\bP}_{z}\big),z\in\overline{\sZ}\big\}
\end{align*}
following the setup in \cite{Bottcher2014}.  The construction of $\wt{\cM}$ proceeds as follows.
\begin{itemize}
\item We let $\wt{\Omega}:=\overline{\bR}_{+}\times\Omega$ to be the new sample space, with elements $\wt{\omega}=(s,\wh{\omega})$, where $s\in\overline{\bR}_{+}$ and $\wh{\omega}\in\wh{\Omega}$. On $\wt{\Omega}$ we consider the $\sigma$-field
    \begin{align*}
    \wt{\sF}:=\Big\{\wt{A}\subset\wt{\Omega}:\,\wt{A}_{s}\in\wh{\sF}_{\infty}^{s}\,\,\text{for any }s\in\overline{\bR}_{+}\Big\},
    \end{align*}
    where $\wt{A}_{s}:=\{\wh{\omega}\in\wh{\Omega}:\,(s,\wh{\omega})\in\wt{A}\}$ and $\wh{\sF}_{\infty}^{s}$ is the last element in $\wh{\bF}_{s}$ (the filtration in $\cM$).
\item We let $\overline{\sZ}=\sZ\cup\{(\infty,\partial,\infty)\}$ to be the new state space, where $\sZ=\bR_{+}\times\sY=\sX\times\bR$, with elements $z=(s,i,a)$. On $\sZ$ we consider the $\sigma$-field
	\begin{align*}
    \wt{\cB}(\sZ):=\left\{\wt{B}\subset\sZ:\,\wt{B}_{s}\in\cB(\sY)\,\,\,\text{for any }s\in\bR_{+}\right\},
	\end{align*}
    where $\wt{B}_{s}:=\big\{(i,a)\in\sY:\,(s,i,a)\in\wt{B}\big\}$. Let $\wt{\cB}(\overline{\sZ}):=\sigma(\wt{\cB}(\sZ)\cup\{(\infty,\partial,\infty)\})$.
\item We consider a family of probability measures $(\wt{\bP}_{z})_{z\in\overline{\sZ}}$, where, for $z=(s,i,a)\in\overline{\sZ}$,
	\begin{align}\label{eq:Probz}
    \wt{\bP}_{z}\big(\wt{A}\big)=\wt{\bP}_{s,i,a}\big(\wt{A}\big):=\wh{\bP}_{s,(i,a)}\big(\wt{A}_{s}\big),\quad\wt{A}\in\wt{\sF}.
	\end{align}
\item We consider the process $Z:=(Z_{t})_{t\in\overline{\bR}_{+}}$ on $(\wt{\Omega},\wt{\sF})$, where, for $t\in\overline{\bR}_{+}$,
	\begin{align}\label{eq:ProcZ}
    Z_{t}(\wt{\omega}):=\big(s+t,\wh{X}_{s+t}(\wh{\omega}),\wh{\varphi}_{s+t}(\wh{\omega})\big),\quad\wt{\omega}=(s,\wh{\omega})\in\wt{\Omega}.
	\end{align}
	Hereafter, we denote the three components of $Z$ by $Z^{1}$, $Z^{2}$, and $Z^{3}$, respectively.
\item On $(\wt{\Omega},\wt{\sF})$, we define  $\wt{\bF}:=(\wt{\sF}_{t})_{t\in\overline{\bR}_{+}}$, where $\wt{\sF}_{t}:=\wt{\sG}_{t+}$ (with the convention $\wt{\sG}_{\infty+}=\wt{\sG}_{\infty}$), and $(\wt{\sG}_{t})_{t\in\overline{\bR}_{+}}$ is the completion of the natural filtration generated by $(Z_{t})_{t\in\overline{\bR}_{+}}$ with respect to the set of probability measures $\{\wt{\bP}_{z},z\in\overline{\sZ}\}$ (cf. \cite[Chapter I]{GikhmanSkorokhod2004}).
\item Finally, for any $r\in\bR_{+}$, we consider the shift operator ${\theta}_{r}:\wt{\Omega}\rightarrow\wt{\Omega}$ defined by
    \begin{align*}
    \theta_{r}\,\wt{\omega}=(u+r,\omega_{\cdot+r}),\quad\wt{\omega}=(u,\omega)\in\wt{\Omega}.
    \end{align*}
    It follows that $Z_{t}\circ{\theta}_{r}=Z_{t+r}$, for any $t,r\in\bR_{+}$.
\end{itemize}

For $z=(s,i,a)\in\overline{\sZ}$, $t\in\overline{\bR}_{+}$, and $\wt{B}\in\wt{\cB}(\overline{\sZ})$, we define the transition function $\wt{P}$ by
\begin{align*}
\wt{P}\big(z,t,\wt{B}\big):=\wt{\bP}_{z}\big(Z_{t}\in\wt{B}\big).
\end{align*}
In view of \eqref{eq:Probz}, we have
\begin{align}\label{eq:TranProbTildeX}
\wt{P}\big(z,t,\wt{B}\big)=\bP_{s,(i,a)}\Big((\wh{X}_{t+s},\wh{\varphi}_{t+s})\in\wt{B}_{s+t}\Big)=P\big(s,(i,a),s+t,\wt{B}_{s+t}\big).
\end{align}
It can be shown that the transition function $\wh{P}$, defined in \eqref{eq:DefTranProbXvarphi}, is associated with a Feller semigroup, so that $\wh{P}$ is a Feller transition function. This and \cite[Theorem 3.2]{Bottcher2014} imply  that $\wt{P}$ is also a Feller transition function. In light of the right continuity of the sample paths, and invoking \cite[Theorem I.4.7]{GikhmanSkorokhod2004}, we conclude that $\wt{\cM}$ is a {\it time-homogeneous strong} Markov family.

%\begin{remark}\textcolor{Blue}{(added 06/21 by Ziteng)}
%Heuristically, we may view $Z$ as a Markov additive process (MAP) as $Z$ is time-homogeneous Markov and for any $t\ge r\ge 0$, the law of $Z^3_{t}-Z^3_{r}$ is determined by $(Z^1_r,Z^2_r)$. We refer to \cite[Chapter XI]{Asmussen2010} and \cite{KyprianouRiveroSengulYang} as well as the reference therein for theory and applications of MAPs. We would like to note that, in our current setup, the MAP is driven by a Markov process on an uncountable state space. Despite the degenerate dynamic of $Z^1$, $Z$ demands tailored analysis that is beyond the scope of extant literature.
%\end{remark}
%
%\trb{My plan was to comment on the relation between our results for (two) exit problems for time-inhomogeneous Markov chains (processes), and the (two) exit problems for MAPs. In particular, to check whether existing results on (two) exit problems for MAPs do not supersede of what is done in our paper.  That is not what the above remark does. So, I guess we remove this remark. For the future reference, we probably should  try to look into the relation between our results for WHf for time-inhomogeneous Markov processes and the WHf for MAPs, as well as the corresponding relation regarding exit problems.
% }

In light of \eqref{eq:DistvarphiInt}, \eqref{eq:Probz}, and \eqref{eq:ProcZ}, for any $(s,i,a)\in\sZ$, we have
\begin{align}\label{eq:DistTildeZ3IntTildeZ2}
\wt{\bP}_{s,i,a}\Big(Z^{3}_{t}=a+\int_{0}^{t}v\big(Z^{2}_{u}\big)\,du,\,\,\text{for all }t\in\bR_{+}\Big)=1.
\end{align}

For any $\ell\in\bR$, we define the auxiliary first passage time
\begin{align*}
\wt{\tau}_{\ell}^{+}:=\inf\big\{t\in\overline{\bR}_{+}:\,Z^{3}_{t}>\ell\big\},\quad\wt{\tau}_{\ell}^{-}:=\inf\big\{t\in\overline{\bR}_{+}:\,Z^{3}_{t}<-\ell\big\}.\nonumber
\end{align*}
Let $-\ell^-\le\ell^+$. We define two auxiliary constrained passage times for $\wt{\cM}$ as
\begin{align}\label{eq:DefTauXiTilde}
\wt{\xi}_{\ell^-,\ell^+}^{+}:=\inf\big\{t\in[0,\wt{\tau}_{-\ell^-}^{-}):\,Z^{3}_{t}>\ell^+\big\},\quad \wt{\xi}_{\ell^-,\ell^+}^{-}:=\inf\big\{t\in[0,\wt{\tau}_{\ell^+}^{+}):\,Z^{3}_{t}<-\ell^-\big\},
\end{align}
which are $\wt{\bF}$-stopping times since $Z^{3}$ has continuous sample paths and $\wt{\bF}$ is right-continuous. By \eqref{eq:ExpPStarExpP}, \eqref{eq:ExpPStarExpPXi}, \eqref{eq:Probz}, \eqref{eq:ProcZ} and \eqref{eq:DistTildeZ3IntTildeZ2}, for any $g^{+}\in \cB_b(\overline{\sX_{+}})$, $(s,i,a)\in\sZ$, and $\ell\in[a,\infty)$,
\begin{align}\label{eq:ExpTildePPStarPlus}
\bE_{s,i}\bigg(g^{+}\bigg(\tau_{\ell-a}^{+},X_{\tau_{\ell-a}^{+}}\bigg)\bigg)=\wt{\bE}_{s,i,a}\Big(g^{+}\Big(Z^{1}_{\wt{\tau}_{\ell}^{+}},Z^{2}_{\wt{\tau}_{\ell}^{+}}\Big)\Big),
\end{align}
which, in particular, implies that
\begin{align}\label{eq:ExpTildePShift}
\wt{\bE}_{s,i,a}\Big(g^{+}\Big(Z^{1}_{\wt{\tau}_{\ell}^{+}},Z^{2}_{\wt{\tau}_{\ell}^{+}}\Big)\Big)=\wt{\bE}_{s,i,0}\Big(g^{+}\Big(Z^{1}_{\wt{\tau}_{\ell-a}^{+}},Z^{2}_{\wt{\tau}_{\ell-a}^{+}}\Big)\Big).
\end{align}
Similarly, for any $g^\pm\in \cB_b(\overline{\sX_\pm})$, $(s,i,a)\in\sZ$ and $-\ell^-\le a\le\ell^+$,%\footnote{.......... comment on $\ell^-$}
\begin{align}\label{eq:ExpTidePStarXiplus}
\bE_{s,i}\Big(g^{+}\Big(\xi_{\ell^-+a,\ell^+-a}^{+},X_{\xi_{\ell^-+a,\ell^+-a}^{+}}\Big)\Big)=\wt{\bE}_{s,i,a}\Big(g^{+}\Big(Z^1_{\wt{\xi}_{\ell^-,\ell^+}^{+}},Z^2_{\wt{\xi}_{\ell^-,\ell^+}^{+}}\Big)\Big).
\end{align}
and
\begin{align}\label{eq:ExpTidePStarXiminus}
\bE_{s,i}\Big(g^{-}\Big(\xi_{\ell^-+a,\ell^+-a}^{-},X_{\xi_{\ell^-+a,\ell^+-a}^{-}}\Big)\Big)=\wt{\bE}_{s,i,a}\Big(g^{-}\Big(Z^1_{\wt{\xi}_{\ell^-,\ell^+}^{-}},Z^2_{\wt{\xi}_{\ell^-,\ell^+}^{-}}\Big)\Big).
\end{align}

We conclude this section with the following lemma. It is an exact adaption of \cite[Lemma 4.2]{BieCheCiaGon2020}, and the proof is therefore omitted here.
\begin{lemma}\label{lem:StrongMarkov}
Let $\wt{\tau}$ be any $\wt{\bF}$-stopping time, and $g^{+}\in \cB_b(\overline{\sX_{+}})$. Then, for any $(s,i,a)\in\overline{\sZ}$ and $\ell\in[a,\infty)$, we have
\begin{align}\label{eq:StrongMarkovCondPlus}
\1_{\{\wt{\tau}\leq\wt{\tau}^{+}_{\ell}\}}\,\wt{\bE}_{s,i,a}\Big(g^{+}\Big(Z^{1}_{\wt{\tau}^{+}_{\ell}},Z^{2}_{\wt{\tau}^{+}_{\ell}}\Big)\,\Big|\,\wt{\sF}_{\wt{\tau}}\Big)=\1_{\{\wt{\tau}\leq\wt{\tau}^{+}_{\ell}\}}\,\wt{\bE}_{Z^{1}_{\wt{\tau}},Z^{2}_{\wt{\tau}},Z^{3}_{\wt{\tau}}}\Big(g^{+}\Big(Z^{1}_{\wt{\tau}^{+}_{\ell}},Z^{2}_{\wt{\tau}^{+}_{\ell}}\Big)\Big),\quad\wt{\bP}_{s,i,a}-\text{a.}\,\text{s.}.
\end{align}
\end{lemma}

\subsection{Proof of Theorem \ref{thm:TwoSidedExit}}\label{subsec:ProofThm}
Let
\begin{align}\label{eq:Defeta}
\wt{\eta}^{+}_{\ell^-,\ell^+}:={\inf\Big\{t\ge \wt{\xi}^{-}_{\ell^-,\ell^+}: Z^3_t>\ell^+\Big\}}.
\end{align}
The auxillary time $\wt{\eta}^{-}_{\ell^-,\ell^+}$ will be used to prove \eqref{eq:Psiplus}. Since the proof of \eqref{eq:Psiminus} can be done in an analogous way, we do not introduce the {`minus'} counterpart of  $\wt{\eta}^{+}_{\ell^-,\ell^+}$. We first establish an important observation.
\begin{lemma}\label{lem:tauhat}
For any $g^+\in \cB_b(\overline{\sX}_+)$ and $-\ell^-\le0\le\ell^+$, we have
\begin{align*}
{\wt{\bE}_{s,i,0}\left(\1_{\set{\wt{\xi}^{-}_{\ell^-,\ell^+}<\infty}}\,g^+\bigg(Z^1_{\wt{\eta}^{+}_{\ell^-,\ell^+}},Z^2_{\wt{\eta}^{+}_{\ell^-,\ell^+}}\bigg)\,\bigg|\,\wt{\sF}_{\wt{\xi}^{-}_{\ell^-,\ell^+}}\right) = \big(J^+\cP^{+}_{\ell^++\ell^-}g^+\big)\bigg(Z^1_{\wt{\xi}^{-}_{\ell^-,\ell^+}},Z^2_{\wt{\xi}^{-}_{\ell^-,\ell^+}}\bigg)}.
\end{align*}
\end{lemma}

\begin{proof}
%Without loss of generality, we assume $\wt{\omega}\in\left\{\wt{\xi}^-_{a,b}<\infty\right\}$; otherwise, both hand sides will be $0$.
Observe that
\begin{align*}
{\1_{\set{\wt{\xi}^{-}_{\ell^-,\ell^+}<\infty}}\,g^+\bigg(Z^{1}_{\wt{\eta}^{+}_{\ell^-,\ell^+}},Z^{2}_{\wt{\eta}^{+}_{\ell^-,\ell^+}}\bigg) = \1_{\set{\wt{\xi}^{-}_{\ell^-,\ell^+}<\infty}}\,g^+\bigg(Z^{1}_{\wt{\tau}^{+}_{\ell^+}},Z^{2}_{\wt{\tau}^{+}_{\ell^+}}\bigg).}
\end{align*}
In addition,
\begin{align*}
{\left\{\wt{\xi}^{-}_{\ell^-,\ell^+}<\infty\right\}=\bigcup_{n\in\bN}\left\{\wt{\xi}^{-}_{\ell^-,\ell^+}<n\right\}\in\wt{\sF}_{\wt{\xi}^{-}_{\ell^-,\ell^+}}\quad\text{and}\quad \left\{\wt{\xi}^{-}_{\ell^-,\ell^+}<\infty\right\}\subseteq\left\{\wt{\xi}^{-}_{\ell^-,\ell^+}\le\wt{\tau}^{+}_{\ell^+}\right\}. }
\end{align*}
The above together with Lemma \ref{lem:StrongMarkov} implies that {
\begin{align*}
&\wt{\bE}_{s,i,0}\left(\1_{\left\{\wt{\xi}^{-}_{\ell^-,\ell^+}<\infty\right\}}\,g^+\bigg(Z^{1}_{\wt{\eta}^{+}_{\ell^-,\ell^+}},Z^{2}_{\wt{\eta}^{+}_{\ell^-,\ell^+}}\bigg)\,\bigg|\,\wt{\sF}_{\wt{\xi}^{-}_{\ell^-,\ell^+}}\right)\\
&\quad= \1_{\left\{\wt{\xi}^{-}_{\ell^-,\ell^+}<\infty\right\}}\1_{\left\{\wt{\xi}^{-}_{\ell^-,\ell^+}\le\wt{\tau}^{+}_{\ell^+}\right\}}\wt{\bE}_{s,i,0}\left(g^+\bigg(Z^{1}_{\wt{\tau}^{+}_{\ell^+}},Z^{2}_{\wt{\tau}^{+}_{\ell^+}}\bigg)\,\bigg|\,\wt{\sF}_{\wt{\xi}^{-}_{\ell^-,\ell^+}}\right)\\
&\quad= \1_{\left\{\wt{\xi}^{-}_{\ell^-,\ell^+}<\infty\right\}}\1_{\left\{\wt{\xi}^{-}_{\ell^-,\ell^+}\le\wt{\tau}^{+}_{\ell^+}\right\}}\wt{\bE}_{Z^{1}_{\wt{\xi}^{-}_{\ell^-,\ell^+}},Z^{2}_{\wt{\xi}^{-}_{\ell^-,\ell^+}},Z^{3}_{\wt{\xi}^{-}_{\ell^-,\ell^+}}}\Big(g^{+}\Big(Z^{1}_{\wt{\tau}^{+}_{\ell^+}},Z^{2}_{\wt{\tau}^{+}_{\ell^+}}\Big)\Big)\\
&\quad= \1_{\left\{\wt{\xi}^{-}_{\ell^-,\ell^+}<\infty\right\}}\,\wt{\bE}_{Z^{1}_{\wt{\xi}^{-}_{\ell^-,\ell^+}},Z^{2}_{\wt{\xi}^{-}_{\ell^-,\ell^+}},Z^{3}_{\wt{\xi}^{-}_{\ell^-,\ell^+}}}\Big(g^{+}\Big(Z^{1}_{\wt{\tau}^{+}_{\ell^+}},Z^{2}_{\wt{\tau}^{+}_{\ell^+}}\Big)\Big).
\end{align*}
Notice} that under $\{\wt{\xi}^{-}_{\ell^-,\ell^+}<\infty\}$, $Z^{3}_{\wt{\xi}^{-}_{\ell^-,\ell^+}}=-\ell^-$ thanks to the continuous sample path of $Z^3$. By \eqref{eq:ExpTildePShift}, {\eqref{eq:ExpTildePPStarPlus}, and} Proposition \ref{prop:ExpgpsimJPPlus}, we have
\begin{align*}
&\1_{\{\wt{\xi}^{-}_{\ell^-,\ell^+}<\infty\}}\,\wt{\bE}_{Z^{1}_{\wt{\xi}^{-}_{\ell^-,\ell^+}},Z^{2}_{\wt{\xi}^{-}_{\ell^-,\ell^+}},Z^{3}_{\wt{\xi}^{-}_{\ell^-,\ell^+}}}\Big(g^{+}\Big(Z^{1}_{\wt{\tau}^{+}_{\ell^+}},Z^{2}_{\wt{\tau}^{+}_{\ell^+}}\Big)\Big)\\
&\quad= \1_{\{\wt{\xi}^{-}_{\ell^-,\ell^+}<\infty\}}\,\wt{\bE}_{Z^{1}_{\wt{\xi}^{-}_{\ell^-,\ell^+}},Z^{2}_{\wt{\xi}^{-}_{\ell^-,\ell^+}},0}\Big(g^{+}\Big(Z^{1}_{\wt{\tau}^{+}_{\ell^{-}+\ell^{+}}},Z^{2}_{\wt{\tau}^{+}_{\ell^{-}+\ell^{+}}}\Big)\Big) = {\Big(J^+\cP^{+}_{\ell^{-}+\ell^{+}}g^+\Big)\bigg(Z^1_{\wt{\xi}^{-}_{\ell^-,\ell^+}},Z^2_{\wt{\xi}^{-}_{\ell^-,\ell^+}}\bigg)}.
\end{align*}
The proof is complete.
\end{proof}

We are ready to prove Theorem \ref{thm:TwoSidedExit}.
\begin{proof}[Proof of Theorem \ref{thm:TwoSidedExit}]
We will only present the proof for $\Xi^+$ as the proof for $\Xi^-$ follows analogously. To start with, in view of \eqref{eq:DefTauXiTilde}, $\set{\wt{\xi}^{+}_{\ell^-,\ell^+}<\infty},\set{\wt{\xi}^{-}_{\ell^-,\ell^+}<\infty}${, and} $\set{\wt{\tau}^{+}_{\ell^+}\wedge\wt{\tau}^{-}_{\ell^-}=\infty}$ form a disjoint partition of $\wt{\Omega}$. It follows that
\begin{align}\label{eq:XiRewrite}
{\wt{\tau}^{+}_{\ell^+} = \wt{\xi}^{+}_{\ell^-,\ell^+}\1_{\left\{\wt{\xi}^{+}_{\ell^-,\ell^+}<\infty\right\}} + \wt{\eta}^+_{\ell^-,\ell^+}\1_{\left\{\wt{\xi}^{-}_{\ell^-,\ell^+}<\infty\right\}} + \infty \1_{\left\{\wt{\tau}^{+}_{\ell^+}\wedge\wt{\tau}^{-}_{\ell^-}=\infty\right\}},}
\end{align}
where we recall the definition of $\wt{\eta}^+_{\ell^-,\ell^+}$ from \eqref{eq:Defeta}. Therefore, for $g^+\in\cB(\overline{\sX_+})$, {
\begin{align}\label{eq:EgDecomp}
\wt{\bE}_{s,i,0}\left(g^+\bigg(Z^1_{\wt{\tau}^{+}_{\ell^+}}, Z^2_{\wt{\tau}^{+}_{\ell^+}}\bigg)\right) &= \wt{\bE}_{s,i,0}\left(\1_{\left\{\wt{\xi}^{+}_{\ell^-,\ell^+}<\infty\right\}}\,g^+\bigg(Z^1_{\wt{\xi}^{+}_{\ell^-,\ell^+}}, Z^2_{\wt{\xi}^{+}_{\ell^-,\ell^+}}\bigg)\right)\nonumber\\
&\quad + \wt{\bE}_{s,i,0}\left(\1_{\left\{\wt{\xi}^{-}_{\ell^-,\ell^+}<\infty\right\}}\,g^+\bigg(Z^1_{\wt{\eta}^+_{\ell^-,\ell^+}}, Z^2_{\wt{\eta}^+_{\ell^-,\ell^+}}\bigg)\right).
\end{align}
For} the left hand side of \eqref{eq:EgDecomp}, due to \eqref{eq:ExpTildePPStarPlus} and Proposition \ref{prop:ExpgpsimJPPlus}, we have
\begin{align}\label{eq:EgDecomp1}
{\wt{\bE}_{s,i,0}\left(g^+\bigg(Z^1_{\wt{\tau}^{+}_{\ell^+}}, Z^2_{\wt{\tau}^{+}_{\ell^+}}\bigg)\right) = \left(\begin{pmatrix}I^+\\ J^+\end{pmatrix}\cP^{+}_{\ell^+} g^+\right)(s,i)},\quad (s,i)\in\sX.
\end{align}
For the first term in the right hand of \eqref{eq:EgDecomp}, due to \eqref{eq:ExpTidePStarXiplus} and \eqref{eq:DefXiplus}, we have
\begin{align}\label{eq:EgDecomp2}
{\wt{\bE}_{s,i,0}\left(\1_{\left\{\wt{\xi}^{+}_{\ell^-,\ell^+}<\infty\right\}}\,g^+\bigg(Z^1_{\wt{\xi}^{+}_{\ell^-,\ell^+}}, Z^2_{\wt{\xi}^{+}_{\ell^-,\ell^+}}\bigg)\right) = \Big(\Xi^{+}_{\ell^-,\ell^+}g^+\Big)(s,i),}\quad (s,i)\in\sX.
\end{align}
For the second term in the right hand side of \eqref{eq:EgDecomp}, due to Lemma \ref{lem:tauhat},  \eqref{eq:ExpTidePStarXiminus} and \eqref{eq:DefXiminus}, we have
\begin{align}\label{eq:EgDecomp3}
{\wt{\bE}_{s,i,0}\left(\1_{\left\{\wt{\xi}^{-}_{\ell^-,\ell^+}<\infty\right\}}\,g^+\bigg(Z^1_{\wt{\eta}^{-}_{\ell^-,\ell^+}}, Z^2_{\wt{\eta}^{-}_{\ell^-,\ell^+}}\bigg)\right) = \Big(\Xi^{-}_{\ell^-,\ell^+}J^+\cP^{+}_{\ell^++\ell^-}g^+\Big)(s,i),} \quad (s,i)\in\sX.
\end{align}
Note also that \eqref{eq:EgDecomp1}, \eqref{eq:EgDecomp1}{, and} \eqref{eq:EgDecomp1} remains true for $(s,i)=(\infty,\partial)$ as both hand sides vanish at coffin state. By combining \eqref{eq:EgDecomp}-\eqref{eq:EgDecomp3}, we yield
\begin{align}\label{eq:JPplusDecomp}
\begin{pmatrix}I^+\\ J^+\end{pmatrix}\cP^{+}_{\ell^+} g^+ = \Xi^{+}_{\ell^-,\ell^+}g^+ + \Xi^{-}_{\ell^-,\ell^+}J^+\cP^{+}_{\ell^++\ell^-}g^+.
\end{align}
By analogue, for $g^-\in \cB_b(\overline{\sX_-})$ we also have
\begin{align}\label{eq:JPminusDecomp}
\begin{pmatrix}J^-\\ I^-\end{pmatrix}\cP^{-}_{\ell^-} g^- =  \Xi^{-}_{\ell^-,\ell^+}g^- + \Xi^{+}_{\ell^-,\ell^+}J^-\cP^{-}_{\ell^{-}+\ell^{+}}g^-.
\end{align}
Substituting $J^+\cP^+_{\ell^{-}+\ell^{+}}g^+$ for $g^-$ in \eqref{eq:JPminusDecomp} then subtracting \eqref{eq:JPminusDecomp} from \eqref{eq:JPplusDecomp}, we yield
\begin{align}\label{eq:IndXiIntermediate}
\left(\begin{pmatrix}I^+\\ J^+\end{pmatrix}\cP^{+}_{\ell^+} - \begin{pmatrix}J^-\\ I^-\end{pmatrix}\cP^{-}_{\ell^-}J^+\cP^+_{\ell^{-}+\ell^{+}}\right) g^+ = \Xi^{+}_{\ell^-,\ell^+}\left(I^+ - J^-\cP^{-}_{\ell^{-}+\ell^{+}}J^+\cP^+_{\ell^{-}+\ell^{+}}\right)g^+.
\end{align}
Below we point out a useful observation, that is, for $N\in\bN$ we have
\begin{align*}
{\Big(I^+-J^-\cP^-_{\ell^{-}+\ell^{+}}J^+\cP^+_{\ell^{-}+\ell^{+}}\Big)\sum_{n=0}^N \Big(J^-\cP^-_{\ell^{-}+\ell^{+}}J^+\cP^+_{\ell^{-}+\ell^{+}}\Big)^n g^+ \!= g^+ \!-\Big(J^-\cP^-_{\ell^{-}+\ell^{+}}J^+\cP^+_{\ell^{-}+\ell^{+}}\Big)^{N+1} g^+.}
\end{align*}
By Lemma \ref{lem:NeumannSum}, letting $N\to\infty$, we yield
\begin{align}\label{eq:NeumannSum}
{\Big(I^+-J^-\cP^-_{\ell^{-}+\ell^{+}}J^+\cP^+_{\ell^{-}+\ell^{+}}\Big)\sum_{k=0}^\infty\Big(J^-\cP^-_{\ell^{-}+\ell^{+}}J^+\cP^+_{\ell^{-}+\ell^{+}}\Big)^n g^+ = g^+,} \quad \ell^\pm\in\bR_+.%\nonumber\\
%& = \sum_{k=0}^\infty (J^-\cP^-_{b-a}J^+\cP^+_{b-a})^k(I^+-J^-\cP^-_{b-a}J^+\cP^+_{b-a}) g^+.
\end{align}
Finally, in view of \eqref{eq:IndXiIntermediate} and \eqref{eq:NeumannSum}, substituting $\sum_{k=0}^\infty (J^-\cP^-_{\ell^{-}+\ell^{+}}J^+\cP^+_{\ell^{-}+\ell^{+}})^k g^+$ for $g^+$ in\eqref{eq:IndXiIntermediate}, the proof is complete.
\end{proof}

\subsection{Proof of Proposition \ref{prop:AfterExit}}
We present below the proof of Proposition \ref{prop:AfterExit}.
\begin{proof}[Proof of Proposition \ref{prop:AfterExit}]
Notice that
\begin{align*}
\bE_{s,i}\left(h(X_T)\1_{\left\{\tau^{-}_{\ell^-}\wedge\tau^{+}_{\ell^+}\le T\right\}}\right) = \bE_{s,i}\left(\bE_{s,i}\left(h(X_T)\1_{\left\{\tau^{-}_{\ell^-}\wedge\tau^{+}_{\ell^+}\le T\right\}}\,\bigg|\,\sF_{\tau^{-}_{\ell^-}\wedge\tau^{+}_{\ell^+}\wedge T}\right)\right).
\end{align*}
Because
\begin{align*}
&\left\{\tau^{-}_{\ell^-}(s)\wedge\tau^{+}_{\ell^+}(s) \le T\right\} \cap \left\{\tau^{-}_{\ell^-}(s)\wedge\tau^{+}_{\ell^+}(s)\wedge T \le t\right\}=\begin{cases}
\left\{\tau^{-}_{\ell^-}(s)\wedge\tau^{+}_{\ell^+}(s) \le t\right\} \in \sF_{t}, & t< T\\
\left\{\tau^{-}_{\ell^-}(s)\wedge\tau^{+}_{\ell^+}(s) \le T\right\} \in \sF_{t}, & t\ge T
\end{cases},
\end{align*}
we have {$\1_{\{\tau^{-}_{\ell^-}(s)\wedge\tau^{+}_{\ell^+}(s) \le T\}}$} is  $\sF_{\tau^{-}_{\ell^-}(s)\wedge\tau^{+}_{\ell^+}(s)\wedge T}$-measurable. Therefore,
\begin{align*}
\bE_{s,i}\left(h(X_T)\1_{\left\{\tau^{-}_{\ell^-}\wedge\tau^{+}_{\ell^+}\le T\right\}}\right) = \bE_{s,i}\left(\1_{\left\{\tau^{-}_{\ell^-}\wedge\tau^{+}_{\ell^+}\le T\right\}}\,{\bE_{s,i}\bigg(h(X_T)\,\bigg|\,\sF_{\tau^{-}_{\ell^-}\wedge\tau^{+}_{\ell^+}\wedge T}\bigg)}\right).
\end{align*}
By \cite[Theorem I.4.6]{GikhmanSkorokhod2004},
\begin{align*}
\bE_{s,i}\left(h(X_T)\1_{\left\{\tau^{-}_{\ell^-}\wedge\tau^{+}_{\ell^+}\le T\right\}}\right) &= \bE_{s,i}\left(\1_{\left\{\tau^{-}_{\ell^-}\wedge\tau^{+}_{\ell^+}\le T\right\}}\bE_{\tau^{-}_{\ell^-}\wedge\tau^{+}_{\ell^+}\wedge T,X_{\tau^{-}_{\ell^-}\wedge\tau^{+}_{\ell^+}\wedge T}}\left(h(X_T)\right)\right)\\
&= \bE_{s,i}\left(\1_{[0,T]}\left(\tau^{-}_{\ell^-}\wedge\tau^{+}_{\ell^+}\right)\bE_{\tau^{-}_{\ell^-}\wedge\tau^{+}_{\ell^+},X_{\tau^{-}_{\ell^-}\wedge\tau^{+}_{\ell^+}}}\left(h(X_T)\right)\right).
\end{align*}
In view of \eqref{eq:SpecialDefgplus} and \eqref{eq:XipmSum}, the proof is complete.
\end{proof}

\section{Concluding remarks and future work}\label{sec:Conclusion}
In this paper, we have shown that certain expectation associated with two-sided exit time can be expressed in terms of expectation operator of one-sided exit time. Our proof is based on the probabilistic decomposition such as \eqref{eq:XiRewrite} and its analytic counterpart \eqref{eq:JPminusDecomp}. This decomposition emerges from the Markov property of an appropriately extended time-space process, introduced in Section \ref{subsec:TimeHomogen}. We would like to highlight that Lemma \ref{lem:NeumannSum} provides a crucial regularity that aids our proof. We conjecture that an analogous regularity holds true in the presence of additive Brownian noise (potentially also driven by the Markov chain), i.e., 
\begin{align*}
\phi_t(s) = \int_s^t v(X_u)d u + \int_s^t\sigma(X_u)d W_u,
\end{align*}
where $X$ and $W$ are independent. Consequently, a similar methodology can be employed to study the expectation associated with the two-sided exit time in this scenario.

%\appendix

\bibliographystyle{alpha}

\begin{thebibliography}{99}

\bibitem[Asm10]{Asmussen2010}
S. Asmussen.
\newblock {\it Applied Probability and Queues}.
\newblock {Springer}, 2010

\bibitem[APU03]{AvramPistoriusUsabel2003}
F. Avram, M. R. Pistorius, and M. Usabel.
\newblock {The two barriers ruin problem via a Wiener Hopf decomposition approach}.
\newblock {\it An. Univ. Craiova Ser. Mat. Inform.}, 30(1):38$-$44, 2003.

\bibitem[BRW80]{BarlowRogersWilliams1980}
M. T. Barlow, L. C. G. Rogers, and D. Williams.
\newblock {Wiener-Hopf factorization for matrices}.
\newblock {\it S\'{e}minaire de Probabilit\'{e}s de Strasbourg}, 14:324$-$331, 1980.

\bibitem[BCCG20]{BieCheCiaGon2020}
T. R. Bielecki, Z. Cheng, I. Cialenco, and R. Gong.
\newblock{Wiener-Hopf factorization technique for time-inhomogeneous finite Markov chains}.
\newblock{\it Stochastics: An International Journal of Probability and Stochastic Processes}, 93(1):130-166,2020.

\bibitem[BCGH20]{BieCiaGonHua2020}
T. R. Bielecki, I. Cialenco, R. Gong, and Y. Huang.
\newblock{Wiener-Hopf factorization for time-inhomogeneous Markov chains and its application}.
\newblock{\it Probability and Mathematical Statistics}, 40(2):225-244, 2020.

\bibitem[Bot14]{Bottcher2014}
B. B\"{o}ttcher.
\newblock {Feller evolution systems: generators and approximation}.
\newblock {\it Stoch. Dynam.}, 14(3), 1350025 (15 pages), 2014.

\bibitem[BSW13]{BottcherSchillingWang2013}
B. B\"{o}ttcher, R. L. Schilling and J. Wang.
\newblock {\it L\'{e}vy Matters III - L\'{e}vy-type processes: construction, approximation and sample path properties}.
\newblock {\it Lecture Notes in Math.}, Vol. 2009, Springer, 2013.

\bibitem[Dyn65]{Dynkin1965}
E. B. Dynkin.
\newblock {\it Markov processes}, Vol. 1.
\newblock {Springer-Verlag Berlin Heidelberg}, Germany, 1965.

\bibitem[ES2015]{EbertStrack2015}
S. Ebert and P. Strack 2015.
\newblock {Until the Bitter End: On Prospect Theory in a Dynamic
Context}.
\newblock{\it American Economic Review}, 105(4), 1618-33, 2015.

\bibitem[EK05]{EthierKurtz2005}
S. N. Ethier and T. G. Kurtz.
\newblock {\it Markov processes characterization and convergence}.
\newblock {John Wiley \& Sons, Inc}, Hoboken, NJ, USA, 2005.

\bibitem[GS04]{GikhmanSkorokhod2004}
I. I. Gikhman and A. V. Skorokhod.
\newblock {\it The theory of stochastic processes}, Vol. II.
\newblock {Springer-Verlag Berlin Heidelberg}, Germany, 2004.

\bibitem[JS03]{JacodShiryaev2003}
J. Jacod and A. N. Shiryaev.
\newblock {\it Limit theorems for stochastic processes}.
\newblock {Springer-Verlag Berlin Heidelberg}, Germany, 2003.

\bibitem[JP08]{JiangPistorius2008}
Z. Jiang and Pistorius.
\newblock {On perpetual American put valuation and first-passage in a regime-switching model with jumps}.
\newblock {\it Finance and Stochastics}, 12(3):331-355, 2008.

\bibitem[KS98]{KaratzasShreve1998}
I. Karatzas and S. E. Shreve.
\newblock {\it Brownian motion and stochasitc calculus}, 2nd Edition.
\newblock {Grad. Texts in Math.}, Vol. 113, Springer, New York, NY, USA, 1998.

\bibitem[KW90]{KennedyWilliams1990}
J. Kennedy and D. Williams.
\newblock {Probabilistic factorization of a quadratic matrix polynomial}.
\newblock {\it Math. Proc. Cambridge Philos. Soc.}, 107(3):591$-$600, 1990.

\bibitem[KRSY20]{KyprianouRiveroSengulYang}
A. Kyprianou, V. Rivero, B. Seng\:ul and T. Yang
\newblock {Entrance law of self-similar Markov processes in high dimensions}.
\newblock {\it Transactions of the American Mathematical Society}, 373(9):6227$-$6299, 2020.

\bibitem[MP11]{MijatovicPistorius2011}
A. Mijatovi\'{c} and M. R. Pistorius.
\newblock {Exotic derivatives under stochasitc volatility models with jumps}.
\newblock {In {\it Adv. Math. Meth. Fin.}}, 455$-$508, Springer-Verlag Berlin Heidelberg, Germany, 2011.


\bibitem[Red2001]{Redner2001}
S. Redner.
\newblock {\it A Guide to First-Passage Processes}.
\newblock {Cambridge University Press}, Cambridge, 2001.


\bibitem[Rog94]{Rogers1994}
L. C. G. Rogers.
\newblock {Fluid models in queueing theory and Wiener-Hopf factorization of markov chains}.
\newblock {\it Ann. Appl. Probab.}, 4(2):390$-$413, 1994.

\bibitem[RW94]{RogersWilliams1994}
L. C. G. Rogers and D. Williams.
\newblock {\it Diffusions, markov processes, and martingales, Vol. I: Foundations}, 2nd Edition.
\newblock {John Wiley \& Sons Ltd}, Chichester, UK, 1994.

\bibitem[RSST99]{RolskiSchmidliSchmidtTeugels1999}
T. Rolski, H. Schmidli, V. Schmidt, and J. Teugels.
\newblock {\it Stochastic processes for insurance and finance}.
\newblock {John Wiley \& Sons Ltd}, Chichester, UK, 1999.

\bibitem[Rud87]{Rudin1987}
W. Rudin.
\newblock {\it Real and Complex Analysis}, 3rd Edition.
\newblock {McGraw-Hill Book Co.}, New York, NY, USA, 1987.

\bibitem[SN2014]{StechmannNeelin2014}
S.N. Stechmann and J.D. Neelin.
\newblock{First-Passage-Time Prototypes for Precipitation Statistics}
\newblock{\it Journal of the Atmospheric Sciences}, 71(9):3269$–$3291, 2014.

\bibitem[Wil08]{Williams2008}
D. Williams.
\newblock {A new look at `Markovian' Wiener-Hopf theory}.
\newblock {In {\it S\'{e}minaire de probabilit\'{e}s XLI}}. {\it Lecture Notes in Math.}, Vol. 1934, 349$-$369, Springer, Berlin, 2008.


\end{thebibliography}

\end{document}